\RequirePackage{fix-cm}
\documentclass{svjour3}                     
\usepackage{graphicx}
\usepackage{mathptmx}

\usepackage{latexsym, epsfig,cite, psfrag,eepic,colordvi,bm}
\usepackage{graphics,color}
\usepackage{amsmath,amsfonts,amssymb,amscd}
\usepackage[all]{xy}

\begin{document}
\title{Plane permutations and applications to a result of Zagier-Stanley and
distances of permutations  
}


\author{Ricky X. F. Chen         \and
        Christian M. Reidys 
}


\institute{Ricky X. F. Chen \and Christian M. Reidys\at
	Biocomplexity Institute and Dept. of Mathematics,
	Virginia Tech,
	1015 Life Science Circle,
	Blacksburg, VA 24061, USA\\
              \email{chen.ricky1982@gmail.com (Ricky X. F. Chen)} \\
              \email{duck@santafe.edu (Christian M. Reidys)}        %
}

\date{Received: date / Accepted: date}

\maketitle

\begin{abstract}

	In this paper, we introduce plane permutations, i.e.~pairs $\mathfrak{p}=(s,\pi)$
	where $s$
	is an $n$-cycle and $\pi$ is an arbitrary permutation,
 represented as a two-row array.
	Accordingly a plane permutation gives rise to three distinct permutations:
	the
	permutation induced by the upper horizontal ($s$), the vertical ($\pi$) and
	the
	diagonal ($D_{\mathfrak{p}}$) of the array. The latter can also be viewed as the three
	permutations of a hypermap. In particular, a map corresponds to a
	plane permutation, in which the diagonal is a fixed point-free involution.
	We
	study the transposition action on plane permutations obtained by permuting
	their diagonal-blocks.
	We establish basic properties of plane permutations and study
	transpositions and
	exceedances and derive various enumerative results. In particular, we prove a recurrence
	for the number
	of plane permutations having a fixed diagonal and $k$ cycles in the
	vertical,
	generalizing Chapuy's recursion for maps filtered by the genus.
	As applications of this framework, we present a combinatorial proof of a
	result of
	Zagier and Stanley, on the number of $n$-cycles $\omega$, for which
	the
	product $\omega(1~2~\cdots ~n)$ has exactly $k$ cycles.
	Furthermore, we integrate studies on the transposition and
	block-interchange distance of permutations as well as the reversal
	distance
	of signed permutations.
	Plane permutations allow us to generalize and recover various lower
	bounds
	for transposition and block-interchange distances and to connect
	reversals
	with block-interchanges.

\end{abstract}

\section{Introduction}

Let $\mathcal{S}_n$ denote the group of permutations, i.e.~the group
of bijections from $[n]=\{1,\dots,n\}$ to $[n]$, where the multiplication
is the composition of maps.
We shall discuss the following three representations of a permutation $\pi$ on $[n]$:\\
\emph{two-line form:} the top line lists all elements in $[n]$, following the natural order.
The bottom line lists the corresponding images of the elements on the top line, i.e.
\begin{eqnarray*}
\pi=\left(\begin{array}{ccccccc}
1&2& 3&\cdots &n-2&{n-1}&n\\
\pi(1)&\pi(2)&\pi(3)&\cdots &\pi({n-2}) &\pi({n-1})&\pi(n)
\end{array}\right).
\end{eqnarray*}
\emph{one-line form:} $\pi$ is represented as a sequence $\pi=\pi(1)\pi(2)\cdots \pi(n-1)\pi(n)$.\\
\emph{cycle form:} regarding $\langle \pi\rangle$ as a cyclic group, we represent $\pi$ by its
collection of orbits (cycles).
The set consisting of the lengths of these disjoint cycles is called the cycle-type of $\pi$.
We can encode this set into a non-increasing integer sequence $\lambda=\lambda_1
\lambda_2\cdots$, where $\sum_i \lambda_i=n$, or as $1^{a_1}2^{a_2}\cdots n^{a_n}$, where
we have $a_i$ cycles of length $i$.
The number of disjoint cycles of $\pi$ will be denoted by $C(\pi)$.
A cycle of length $k$ will be called a $k$-cycle. A cycle of odd and even length will be called an odd
and even cycle, respectively. It is well known that all permutations of a same cycle-type
form a conjugacy class of $\mathcal{S}_n$.

Zagier~\cite{zag} and Stanley~\cite{stan3} studied the following problem:
how many permutations $\omega$ from a fixed conjugacy class of $\mathcal{S}_n$ such that
the product $\omega (1~2~\cdots ~n)$ has exactly $k$ cycles?

Both authors employed the character theory of the symmetric group in order to obtain
certain generating polynomials. Then,
by evaluating these polynomials at specific conjugacy classes,
Zagier obtained an explicit formula for the number of rooted one-face maps (i.e.,
the conjugacy class consists of involutions without fixed points), and
both, Zagier as well as Stanley, obtained the following surprisingly simple
formula for the conjugacy class $n^1$:
the number $\xi_{1,k}(n)$ of $\omega$ for which $\omega (1~2~\cdots ~n)$ has exactly $k$ cycles is $0$ if
$n-k$ is odd, and otherwise $\xi_{1,k}(n)=\frac{2C(n+1,k)}{n(n+1)}$
where $C(n,k)$ is the unsigned Stirling number of the first kind,
i.e., the number of permutations on $[n]$ with $k$ cycles.
Stanley asked for a combinatorial proof for this result~\cite{stan3}.
Such proofs were later given in~\cite{feray} and in~\cite{schaef}.
In this paper, we will give another combinatorial proof, using the framework of plane permutations.
For this purpose, we will study exceedances
via a natural transposition action on plane permutations.

The transposition action on plane permutations has also direct connections to various distances
of permutations and signed permutations.
This ties to important problems in the context of bioinformatics,
in particular the evolution of genomes by rearrangements in DNA as well as RNA.
For the related studies and general biological background, we refer to
~\cite{pev1,pev2,pev3,pev4,chri1,chri2,bul,braga,eli,yan} and the references therein.

{
An outline of this paper is as follows. 
In Section $2$, we develop our framework.
We introduce plane permutations, $\mathfrak{p}=(s,\pi)$, study transpositions, exceedances, establish basic
properties and derive various enumerative results.

}

In Section $3$, as the first application of the plane permutation framework, we present our proof for the result of Zagier and Stanley mentioned above.
{ To this end, we view ordinary permutations as a particular class of plane permutations and classify them by their diagonals.
We combinatorially prove a new recurrence
satisfied by the unsigned Stirling numbers of the first kind which is 
the same recurrence for $\xi_{1,k}(n)$ derived from one of the obtained recurrences on 
plane permutations so that the Zagier-Stanley result follows.}

{ 
	
	In section $4$ and $5$,
	we study the transposition distance and the block-interchange distance of permutations,
	respectively.
	We derive general lower bounds in the form of optimizing
	a free parameter succinctly from two lemmas regarding 
	the transposition action on plane permutations. 
	This is different from the existing graph model approach~\cite{pev1,pev2,pev3,chri1} and the permutation group theory approach
	~\cite{feimei,hhtl,huanglu,labarre2,Meid}.
	The existing lower bounds, e.g., Bafna and Pevzner~\cite{pev1}, and Christie~\cite{chri2},
	can be refined by a particular choice of the free parameter.
	Our formula of the lower bound motivates several
	optimization problems as well. We will completely solve one of them that is to determine $\max_{\gamma} |C(\alpha\gamma)-C(\gamma)|$
	for a fixed permutation $\alpha$ when $\gamma$ ranges over all permutations.

	In Section $6$, we study the reversal distance of signed permutations.
	By translating the reversal distance of signed
	permutations into block-interchange distance of permutations with restricted block-interchanges, we
	prove a new formula for a lower bound of the reversal distance.
	We then observe that this bound is typically equal to the reversal distance.

}


\section{Plane permutations}

In this section, we will introduce plane permutations and present basic results on that.

\begin{definition}[Plane permutation]
A plane permutation on $[n]$ is a pair $\mathfrak{p}=(s,\pi)$ where $s=(s_i)_{i=0}^{n-1}$
is an $n$-cycle and $\pi$ is an arbitrary permutation on $[n]$. 
\end{definition}\label{2def1}

 Given $s=(s_0~s_1~\cdots ~s_{n-1})$,
a plane permutation $\mathfrak{p}=(s,\pi)$ is represented by a two-row array:
\begin{equation}
\mathfrak{p}=\left(\begin{array}{ccccc}
s_0&s_1&\cdots &s_{n-2}&s_{n-1}\\
\pi(s_0)&\pi(s_1)&\cdots &\pi(s_{n-2}) &\pi(s_{n-1})
\end{array}\right).
\end{equation}
The permutation $D_{\mathfrak{p}}$ induced by the diagonal-pairs (cyclically) in the array, i.e., $D_{\mathfrak{p}}(\pi(s_{i-1}))=s_i$ for $0<i< n$, and
$D_{\mathfrak{p}}(\pi(s_{n-1}))=s_0$, is called the diagonal of $\mathfrak{p}$.

{\bf Observation:} $D_{\mathfrak{p}}=
s \pi^{-1}$.

In a permutation $\pi$ on $[n]$, $i$ is called an exceedance if $i<\pi(i)$ and an anti-exceedance otherwise.
Note that $s$ induces a partial order $<_s$,
where $a<_s b$ if $a$ appears before $b$ in $s$ from left to right (with the left most element $s_0$).
These concepts then can be generalized for plane permutations as follows:
\begin{definition}
For a plane permutation $\mathfrak{p}=(s,\pi)$, an element $s_i$ is called an
exceedance of $\mathfrak{p}$ if $s_i<_s \pi(s_i)$, and an anti-exceedance if $s_i\ge_s \pi(s_i)$.
\end{definition}

In the following, we mean by ``the cycles of $\mathfrak{p}=(s,\pi)$'' the cycles of $\pi$ and
any comparison of elements in $s,~\pi$ and $D_{\mathfrak{p}}$ references $<_s$.

Obviously, each $\mathfrak{p}$-cycle contains at least one anti-exceedance as it contains
a minimum, $s_i$, for which $\pi^{-1}(s_i)$ will be an anti-exceedance. We call these trivial anti-exceedances
and refer to a non-trivial anti-exceedance as an NTAE. Furthermore, in any cycle of length
greater than one, its minimum is always an exceedance.

It should be easy for the reader to check that
the number of exceedances of $\mathfrak{p}$ does not depend on how we write $s$ in the top row in the two-row representation of $\mathfrak{p}$ although the set of exceedances may vary according to different cyclic shift of $s$.
Let $Exc(\mathfrak{p})$ and $AEx(\mathfrak{p})$ denote the number of exceedances and
anti-exceedances of $\mathfrak{p}$, respectively. For $D_{\mathfrak{p}}$, the quantities
$Exc(D_{\mathfrak{p}})$ and $AEx(D_{\mathfrak{p}})$ are defined in reference to $<_s$.

\begin{lemma}\label{2lem1}
For a plane permutation $\mathfrak{p}=(s,\pi)$, we have
\begin{equation}
Exc(\mathfrak{p})=AEx(D_{\mathfrak{p}})-1.
\end{equation}
\end{lemma}
\begin{proof} By construction of the diagonal permutation $D_{\mathfrak{p}}$, we have
$$
\forall\;  0\leq i<n-1, \quad s_i<_s\pi(s_i)\qquad \Longleftrightarrow \qquad
                           \pi(s_i)\geq_s D_{\mathfrak{p}}(\pi(s_i))=s_{i+1}.
$$
Note that $s_{n-1}$ is always an anti-exceedance of $\mathfrak{p}$ since
$s_{n-1} \ge \pi(s_{n-1})$, and that $\pi(s_{n-1})$ is always an anti-exceedance of
$D_{\mathfrak{p}}$ since $D_{\mathfrak{p}}(\pi(s_{n-1}))=s(s_{n-1})=s_0$ and
$\pi(s_{n-1}) \ge s_0$. Thus we have
$$
Exc(\mathfrak{p})=AEx(D_{\mathfrak{p}})-1,
$$
whence the lemma.\qed
\end{proof}

\begin{proposition}\label{2pro1}
For a plane permutation $\mathfrak{p}=(s,\pi)$ on $[n]$, the sum of the number of
cycles in $\pi$ and in $D_{\mathfrak{p}}$ is smaller than $n+2$.
\end{proposition}
\begin{proof} Since each cycle has at least one anti-exceedance,
we have $AEx(\mathfrak{p})\geq C(\pi)$ and $AEx(D_{\mathfrak{p}})\geq C(D_{\mathfrak{p}})$. Using Lemma~\ref{2lem1},
$$
AEx(\mathfrak{p}) = n-Exc(\mathfrak{p}) = n+1-AEx(D_{\mathfrak{p}}) \geq C(\pi).
$$
Therefore,
$$
n+1 \geq C(\pi) + AEx(D_{\mathfrak{p}}) \geq C(\pi)+C(D_{\mathfrak{p}}),
$$
whence the proposition.\qed
\end{proof}

In fact, based on Proposition~\ref{2pro1}, it can be proved that the maximum $n+1$ is attained for any given $\pi$,
see~\cite{ChR1} for instance.

\begin{proposition}\label{2pro2}
For a plane permutation $\mathfrak{p}=(s,\pi)$ on $[n]$, the quantities $C(\pi)$ and
$C(D_{\mathfrak{p}})$ satisfy
\begin{equation}
C(\pi)+C(D_{\mathfrak{p}})\equiv n-1 \pmod{2}.
\end{equation}
\end{proposition}
\begin{proof} In view of $s=D_{\mathfrak{p}}\pi$, the parity of both sides are equal.
Since a $k$-cycle can be written as a product of $k-1$ transpositions, the parity of the
LHS is the same as $n-1$ while the parity of the RHS is the same as $(n-C(\pi)) + (n-C(D_\mathfrak{p}))$,
whence the proposition.\qed
\end{proof}

Given a plane permutation $(s,\pi)$ on $[n]$ and a sequence $h=(i,j,k,l)$, such that $i\leq j<k \leq l$
and $\{i,j,k,l\}\subset [n-1]$, let
$$
s^h=(s_0 ~s_1 ~\dots ~s_{i-1} ~\underline{s_k ~\dots ~s_l} ~s_{j+1} ~\dots ~s_{k-1} ~
\underline{s_i ~\dots ~s_j} ~s_{l+1} ~\dots),
$$
i.e.~the $n$-cycle obtained by transposing the blocks $[s_i,s_j]$ and $[s_k,s_l]$ in $s$.
Note that in case of $j+1=k$, we have
$$
s^h=(s_0 ~s_1 ~\dots ~s_{i-1} ~\underline{s_k ~\dots ~s_l} ~\underline{s_i ~\dots ~s_j} ~s_{l+1} ~\dots).
$$
Let furthermore
$$
\pi^h=D_{\mathfrak{p}}^{-1} s^h,
$$
that is, the derived plane permutation, $(s^h,\pi^h)$, can be represented as

\begin{eqnarray*}
\left(
\vcenter{\xymatrix@C=0pc@R=1pc{
\cdots s_{i-1}\ar@{->>}[d]  & s_k\ar@{--}[dl] &\cdots & s_l\ar@{--}[dl]\ar@{->>}[d] & s_{j+1} &\hspace{-0.5ex}\cdots\hspace{-0.5ex} &
s_{k-1}\ar@{->>}[d] & s_i\ar@{--}[dl] &\cdots & s_{j}\ar@{--}[dl]\ar@{->>}[d] & s_{l+1}  \cdots\\
\cdots \pi(s_{k-1}) & \pi(s_k) & \cdots & \pi(s_j) & \pi(s_{j+1}) &\hspace{-0.5ex}\cdots\hspace{-0.5ex} & \pi(s_{i-1}) & \pi(s_i) &\cdots & \pi(s_l)& \pi(s_{l+1}) \cdots
}}
\right).
\end{eqnarray*}
We write $(s^h,\pi^h)= \chi_h\circ(s,\pi)$.
Note that the bottom row of the two-row representation of $(s^h,\pi^h)$
is obtained by transposing the blocks $[\pi(s_{i-1}),\pi(s_{j-1})]$ and
$[\pi(s_{k-1}),\pi(s_{l-1})]$ of the bottom row of $(s,\pi)$.
In the following, we refer to general $\chi_h$ as block-interchange and for the special case of $k=j+1$,
we refer to $\chi_h$ as transpose.
As a result, we observe
\begin{lemma}\label{2lemx1}
Let $(s,\pi)$ be a plane permutation on $[n]$ and $(s^h,\pi^h)=\chi_h \circ (s,\pi)$ for $h=(i,j,k,l)$.
Then, $\pi(s_r)=\pi^h(s_r)$ if $r\in \{0,1,\ldots,n-1\}\setminus \{i-1,j,k-1,l\}$. Moreover,
for $j+1<k$
\begin{equation*}
\pi^h(s_{i-1})=\pi(s_{k-1}), \quad \pi^h(s_j)=\pi(s_l), \quad \pi^h(s_{k-1})=\pi(s_{i-1}), \quad
\pi^h(s_l)=\pi(s_j),
\end{equation*}
and for $j=k-1$, we have
$$
\pi^h(s_{i-1})=\pi(s_{j}), \quad \pi^h(s_j)=\pi(s_l), \quad \pi^h(s_l)=\pi(s_{i-1}).
$$
\end{lemma}

We shall proceed by analyzing the induced changes of the $\pi$-cycles when passing
to $\pi^h$. By Lemma~\ref{2lemx1}, only the $\pi$-cycles containing $s_{i-1}$, $s_{j}$,
$s_l$ will be affected so that only these changes will be explicitly displayed.

\begin{lemma}\label{2lem4}
Let $(s^h, \pi^h)=\chi_h \circ (s, \pi)$, where $h=(i,j,j+1,l)$.
Then there exist the following six possible scenarios for the pair $(\pi,\pi^h)$:

\begin{center}
\begin{tabular}{|c|c|c|}
\hline
Case~$1$ &$\pi$& $(s_{i-1} ~v_1^i ~\ldots v_{m_i}^i)  (s_j ~v_1^j ~\ldots v_{m_j}^j) (s_l ~v_1^l ~\ldots v_{m_l}^l)$\\\cline{2-3}
&$\pi^h$&$ (s_{i-1} ~v_1^j ~\ldots v_{m_j}^j  ~s_{j}  ~v_1^l ~\ldots v_{m_l}^l ~s_l ~v_1^i ~\ldots v_{m_i}^i)$\\
\hline\hline
Case~$2$ &$\pi$& $(s_{i-1} ~v_1^i  ~\ldots v_{m_i}^i ~s_l ~v_1^l ~\ldots ~v_{m_l}^l ~s_j  ~v_1^j ~\ldots ~v_{m_j}^j)$\\\cline{2-3}
&$\pi^h$& $(s_{i-1} ~v_1^j ~\ldots ~v_{m_j}^j)  (s_j ~v_1^l ~\ldots ~v_{m_l}^l) (s_l ~v_1^i ~\ldots ~v_{m_i}^i)$\\
\hline\hline
Case~$3$ &$\pi$& $(s_{i-1} ~v_1^i ~\ldots ~v_{m_i}^i ~s_j ~v_1^j ~\ldots ~v_{m_j}^j ~s_l ~v_1^l ~\ldots ~v_{m_l}^l)$\\\cline{2-3}
&$\pi^h$& $(s_{i-1} ~v_1^j ~\ldots ~v_{m_j}^j ~s_l ~v_1^i ~\ldots ~v_{m_i}^i ~s_j ~v_1^l ~\ldots ~v_{m_l}^l)$\\
\hline\hline
Case~$4$ &$\pi$& $(s_{i-1} ~v_1^i ~\ldots ~v_{m_i}^i ~s_{j} ~v_1^j ~\ldots ~v_{m_j}^j) (s_l ~v_1^l ~\ldots ~v_{m_l}^l)$\\\cline{2-3}
&$\pi^h$&$ (s_{i-1} ~v_1^j ~\ldots ~v_{m_j}^j) (s_{j} ~v_1^l ~\ldots ~v_{m_l}^l ~s_l ~v_1^i ~\ldots ~v_{m_i}^i)$\\
\hline\hline
Case~$5$ &$\pi$& $(s_{i-1} ~v_1^i ~\ldots ~v_{m_i}^i)  (s_{j} ~v_1^j ~\ldots ~v_{m_j}^j ~s_l ~v_1^l ~\ldots ~v_{m_l}^l)$\\\cline{2-3}
&$\pi^h$&$ (s_{i-1} ~v_1^{j} ~\ldots ~v_{m_j}^j ~s_l ~v_1^i ~\ldots ~v_{m_i}^i) (s_{j} ~v_1^l ~\ldots ~v_{m_l}^l)$\\
\hline\hline
Case~$6$ &$\pi$& $(s_{i-1} ~v_1^i ~\ldots ~v_{m_i}^i ~s_l ~v_1^l ~\ldots ~v_{m_l}^l) (s_{j} ~v_1^j ~\ldots ~v_{m_j}^j)$\\\cline{2-3}
&$\pi^h$&$ (s_{i-1} ~v_1^j ~\ldots ~v_{m_j}^j ~s_{j} ~v_1^l ~\ldots ~v_{m_l}^l) (s_l ~v_1^i ~\ldots ~v_{m_i}^i)$\\
\hline
\end{tabular}
\end{center}
\end{lemma}
\begin{proof} We shall only prove Case~$1$ and Case~$2$, the remaining four cases
can be shown analogously.
For Case~$1$, the $\pi$-cycles containing $s_{i-1}$,~$s_j$,~$s_l$ are
$$
(s_{i-1} ~v_1^i ~\ldots ~v_{m_i}^i),  (s_j ~v_1^j ~\ldots ~v_{m_j}^j), (s_l ~v_1^l ~\ldots ~v_{m_l}^l).
$$
Lemma~\ref{2lemx1} allows us to identify the new cycle structure by inspecting the critical
points $s_{i-1}$, $s_j$ and $s_l$.
Here we observe that all three cycles merge and form a single $\pi^h$-cycle
\begin{eqnarray*}
(s_{i-1} ~\pi^h(s_{i-1}) ~(\pi^h)^2(s_{i-1}) ~\ldots)&=&(s_{i-1} ~\pi(s_j) ~\pi^2(s_j) ~\ldots )\\
&=& (s_{i-1} ~v_1^j ~\ldots ~v_{m_j}^j ~s_{j}  ~v_1^l ~\ldots ~v_{m_l}^l ~s_l ~v_1^i ~\ldots ~v_{m_i}^i).
\end{eqnarray*}
For Case $2$, the $\pi$-cycle containing $s_{i-1}$,~$s_j$,~$s_l$ is
$$
(s_{i-1} ~v_1^i ~\ldots ~v_{m_i}^i  ~s_l ~v_1^l ~\ldots ~v_{m_l}^l ~s_j ~v_1^j ~\ldots ~v_{m_j}^j).
$$
We compute the $\pi^h$-cycles containing $s_{i-1}$, $s_j$ and $s_l$ in $\pi^h$ as
\begin{eqnarray*}
	(s_j ~\pi^h(s_j) ~(\pi^h)^2(s_j) ~\ldots)&=&(s_j ~\pi(s_{l}) ~\pi^2(s_{l}) ~\ldots)=(s_j ~v_1^l ~\ldots ~v_{m_l}^l)\\
	(s_l ~\pi^h(s_l) ~(\pi^h)^2(s_l) ~\ldots)&=&(s_l ~\pi(s_{i-1}) ~\pi^2(s_{i-1}) ~\ldots)=(s_l ~v_1^i ~\ldots ~v_{m_i}^i)\\
(s_{i-1} ~\pi^h(s_{i-1}) ~(\pi^h)^2(s_{i-1}) ~\ldots)&=&(s_{i-1} ~\pi(s_{j})
~\pi^2(s_{j}) ~\ldots)=(s_{i-1} ~v_1^j ~\ldots ~v_{m_j}^j)
\end{eqnarray*}
whence the lemma.
\qed
\end{proof}

If we wish to express which cycles are impacted by a transpose of scenario $k$ acting on a plane
permutation, we shall say ``the cycles are acted upon by a Case~$k$ transpose''.

We next observe
\begin{lemma}\label{2lem3}
Let $\mathfrak{p}^h=\chi_h \circ \mathfrak{p}$ where $\chi_h$ is a transpose.
Then the difference of the number of cycles of $\mathfrak{p}$ and $\mathfrak{p}^h$ is even.
Furthermore the difference of the number of cycles, odd cycles, even cycles between
$\mathfrak{p}$ and $\mathfrak{p}^h$ is contained in $\{-2,0,2\}$.
\end{lemma}
\begin{proof}
Lemma~\ref{2lem4} implies that the difference of the numbers of cycles of $\pi$ and $\pi^h$
is even.
As for the statement about odd cycles, since the parity of the total number of elements
contained in the cycles containing $s_{i-1}$, $s_{j}$ and $s_l$ is preserved, the difference of the
number of odd cycles is even. Consequently, the difference of the number of even cycles is also
even whence the lemma.
\qed
\end{proof}

Suppose we are given $h=(i,j,k,l)$, where $j+1<k$. Then using the strategy of the proof
of Lemma~\ref{2lem4}, we have
\begin{lemma}\label{2lem5}
Let $(s^h, \pi^h)=\chi_h \circ (s, \pi)$, where $h=(i,j,k,l)$ and $j+1<k$.
Then, the difference of the numbers of $\pi$-cycles and $\pi^h$-cycles is contained in
$\{-2,0,2\}$. Furthermore, the scenarios, where the number of $\pi^h$-cycles increases by $2$, are
given by:
\begin{center}
\begin{tabular}{|c|c|c|}
\hline
Case~$a$ &$\pi$& $(s_{i-1} ~v_1^i ~\ldots ~v_{m_i}^i ~s_j ~v_1^j ~\ldots ~v_{m_j}^j ~s_{l} ~v_1^l ~\ldots ~v_{m_l}^l ~s_{k-1} ~v_1^k ~\ldots ~v_{m_k}^k)$\\\cline{2-3}
&$\pi^h$&$(s_{i-1} ~v_1^k ~\ldots ~v_{m_k}^k) (s_j ~v_1^l ~\ldots ~v_{m_l}^l ~s_{k-1} ~v_1^i ~\ldots ~v_{m_i}^i) (s_l ~v_1^j ~\ldots ~v_{m_j}^j) $\\
\hline\hline
Case~$b$ &$\pi$& $(s_{i-1} ~v_1^i ~\ldots ~v_{m_i}^i ~s_{k-1}  ~v_1^k ~\ldots ~v_{m_k}^k ~s_{j}  ~v_1^j ~\ldots ~v_{m_j}^j ~s_l ~v_1^l ~\ldots ~v_{m_l}^l)$\\\cline{2-3}
&$\pi^h$&$(s_{i-1} ~v_1^k ~\ldots ~v_{m_k}^k ~s_j ~v_1^l ~\ldots ~v_{m_l}^l)  (s_{k-1} ~v_1^i ~\ldots ~v_{m_i}^i) (s_l ~v_1^j ~\ldots ~v_{m_j}^j) $\\
\hline\hline
Case~$c$ &$\pi$& $(s_{i-1} ~v_1^i ~\ldots ~v_{m_i}^i ~s_{k-1}  ~v_1^k ~\ldots ~v_{m_k}^k ~s_{l}  ~v_1^l  ~\ldots ~v_{m_l}^l ~s_j  ~v_1^j  ~\ldots  ~v_{m_j}^j)$\\\cline{2-3}
&$\pi^h$&$(s_{i-1}  ~v_1^k  ~\ldots ~v_{m_k}^k  ~s_l ~v_1^j  ~\ldots v_{m_j}^j)  (s_{k-1} ~v_1^i ~\ldots ~v_{m_i}^i)  (s_j  ~v_1^l ~\ldots ~v_{m_l}^l) $\\
\hline\hline
Case~$d$ &$\pi$& $(s_{i-1}  ~v_1^i  ~\ldots ~v_{m_i}^i  ~s_l  ~v_1^l  ~\ldots ~v_{m_l}^l ~s_{j}  ~v_1^j  ~\ldots ~v_{m_j}^j ~s_{k-1}  ~v_1^k ~\ldots  ~v_{m_k}^k)$\\\cline{2-3}
&$\pi^h$&$(s_{i-1} ~v_1^k  ~\ldots ~v_{m_k}^k)  (s_j  ~v_1^l  ~\ldots ~v_{m_l}^l) (s_{k-1} ~v_1^i  ~\ldots ~v_{m_i}^i  ~s_l ~v_1^j  ~\ldots ~v_{m_j}^j) $\\
\hline\hline
Case~$e$ &$\pi$& $(s_{i-1} ~v_1^i ~\ldots ~v_{m_i}^i ~s_{k-1} ~v_1^k  ~\ldots ~v_{m_k}^k)  (s_{j} ~v_1^j  ~\ldots ~v_{m_j}^j ~s_l ~v_1^l  ~\ldots ~v_{m_l}^l)$\\\cline{2-3}
&$\pi^h$&$(s_{i-1}  ~v_1^k  ~\ldots ~v_{m_k}^k) (s_j ~v_1^l  ~\ldots ~v_{m_l}^l)  (s_{k-1}   ~v_1^i  ~\ldots ~v_{m_i}^i)  (s_l ~v_1^j ~\ldots ~v_{m_j}^j) $\\
\hline
\end{tabular}
\end{center}

\end{lemma}

\begin{definition}
Two plane permutations $(s,\pi)$ and $(s',\pi')$ on $[n]$ are equivalent if there exists a
permutation $\alpha$ on $[n]$ such that
$$
s=\alpha s' \alpha^{-1}, \quad \pi=\alpha\pi'\alpha^{-1}.
$$
\end{definition}

\begin{lemma}\label{2lem8}
For two equivalent plane permutations $\mathfrak{p}=(s,\pi)$ and $\mathfrak{p}'=(s',\pi')$,
we have
\begin{eqnarray}
Exc(\mathfrak{p})=Exc(\mathfrak{p}').
\end{eqnarray}
\end{lemma}

\begin{proof}
Assume
$
s=\alpha s' \alpha^{-1}, \pi=\alpha\pi'\alpha^{-1}
$
for some $\alpha$.
Since conjugation by $\alpha$ is equivalent to relabeling according to $\alpha$, $a<_{s'} b$
implies $\alpha(a)<_s \alpha(b)$. Therefore, an exceedance of $\mathfrak{p}'$ will uniquely
correspond to an exceedance of $\mathfrak{p}$, whence the lemma. \qed
\end{proof}

{
	
	Let $q^{\lambda}$ denote the number of permutations being of cycle-type $\lambda$.
	Given a permutation $\gamma$ with cycle-type $\lambda$, denote $W_{\mu,\eta}^{\lambda}$
	the number of different ways of writing $\gamma$ as a product of $\alpha$ and $\beta$, i.e.,
	$\gamma=\alpha \beta$, where $\alpha$ is of cycle-type $\mu$ and $\beta$ is of cycle-type $\eta$.
	Clearly, this number only depends on $\lambda$ instead of specific choice of $\gamma$.
	Also, we have:
	$$
	W_{\mu,\eta}^{\lambda}=W_{\eta,\mu}^{\lambda}, \quad
	q^{\lambda}W_{\mu,\eta}^{\lambda}= q^{\mu}W_{\lambda,\eta}^{\mu}.
	$$
	
	Let $U_{D}$ denote the set of plane permutations having
	$D$ as diagonals for some fixed permutation $D$ on $[n]$ of cycle-type $\lambda$.
	Note $\mathfrak{p}=(s,\pi)\in U_D$ iff $D=D_{\mathfrak{p}}=s\circ \pi^{-1}$.
	Then, the number $|U_D|$ enumerates the ways to write $D$ as a
	product of an $n$-cycle with another permutation.
	%
	%
	%
	Due to symmetry,
	$|U_D|$ is also certain multiple of the number
	of factorizations of $(1~2~\cdots ~n)$ into a permutation of cycle-type $\lambda$ and another permutation,
	i.e., rooted hypermaps having one face.
	A rooted hypermap is a triple of permutations $(\alpha,\beta_1,\beta_2)$,
	such that $\alpha=\beta_1\beta_2$. The cycles in $\alpha$ are called faces, the cycles in $\beta_1$
	are called (hyper)edges, and the cycles in $\beta_2$ are called vertices. If $\beta_1$ is
	an involution without fixed points, the rooted hypermap is an ordinary rooted map.
	We refer to~\cite{chap,walsh1, walsh2,lzv,hz,zv,IJ,jac,reidys,ChR1} and references therein for an
	in-depth study of hypermaps and maps.
	
	Plane permutations in two-row arrays can be viewed as a new way to represent one-face hypermaps.
	However, there are
	some advantages to deal with this new representation.
	As a quick application, we prove the cornerstone, i.e., the trisection lemma, in Chapuy~\cite{chap}
	where a new recurrence
	satisfied by the number of rooted one-face maps of genus $g$ was obtained.
	A rooted map with $n$ edges, or equivalently, a plane permutation
	$\mathfrak{p}=(s,\pi)$ on $[2n]$ such that $D_{\mathfrak{p}}$
	is an involution without fixed points, is of genus $g$, just means that $\pi$ has
	$n+1-2g$ cycles.
	In~\cite{chap}, the concepts of
	up-step, down-step and trisection of one-face maps were defined. These concepts are respectively
	the same as exceedance, anti-exceedance
	and NTAE of plane permutations whose diagonals are involutions without fixed points.
	Then, the trisection lemma can be restated as follows:
	\begin{lemma}[The trisection lemma~\cite{chap}]
		There are $2g$ NTAEs in a rooted one-face map with $n$ edges and genus $g$.
	\end{lemma}
	This can be easily seen in the following way: given a rooted one-face map $\mathfrak{p}=(s,\pi)$,
	$D_{\mathfrak{p}}$ has always $n$ exceedances and $n$ anti-exceedances irrespective of $<_s$
	since it is an involution without fixed points. By Lemma~\ref{2lem1},
	$\mathfrak{p}$ has $n+1$ anti-exceedances.
	Therefore, $\mathfrak{p}$ has $(n+1)-(n+1-2g)=2g$ NTAEs since $\pi$ has $n+1-2g$ cycles.

	Next, we shall enumerate plane permutations in $U_D$
	having $k$ cycles and $a$ exceedances, where $D$ is a fixed permutation of cycle-type $\lambda$.
	\begin{lemma}\label{2lem6}
		Let $C_1$ and $C_2$ be two $\pi$-cycles of $(s,\pi)$ such that $\min\{C_1\}<_s\min\{C_2\}$.
		Suppose we have a Case~$2$ transpose on $C_2$, splitting $C_2$ into the three
		$\pi^h$-cycles $C_{21},C_{22},C_{23}$ in $(s^h,\pi^h)$. Then
		\begin{equation}
		\min\{C_1\}<_{s^h}\min\{\min\{C_{21}\},\min\{C_{22}\},\min\{C_{23}\}\}.
		\end{equation}
	\end{lemma}
	\begin{proof} Note that any Case~$2$ transpose on $C_2$ will not change $C_1$.
		Furthermore, it will only impact the relative order of elements larger than $\min\{C_2\}$,
		whence the proof.\qed
	\end{proof}

	Let $Y_1$ denote the set of pairs $(\mathfrak{p},\epsilon)$, where
	$\mathfrak{p}\in U_D$ has $b$ cycles and $\epsilon$ is an NTAE
	in $\mathfrak{p}$.
	Let furthermore $Y_2$ denote the set of $\mathfrak{p'}\in U_D$
	in which there are $3$ labeled cycles among the total $b+2$ $\mathfrak{p'}$-cycles and
	finally let $Y_3$ denote the set of plane permutations $\mathfrak{p'}\in
	U_D$ where there are $3$ labeled cycles among the total $b+2$ $\mathfrak{p'}$-cycles
	and a distinguished NTAE contained in the labeled cycle that contains the largest minimal
	element.
	
	We will show $|Y_1|=|Y_2|+|Y_3|$ for any $D$ by establishing a bijection for plane permutations based on Case $1$ and Case $2$ of
	Lemma~\ref{2lem4}. This bijection is motivated by the gluing/slicing bijection of Chapuy~\cite{chap} for maps
	(i.e., $D$ is restricted to be an involution without fixed points).
	In fact, Case $1$ corresponds to the gluing operation and Case $2$ corresponds to the slicing
	operation.
	Our results extend those of~\cite{chap} to hypermaps as gluing/slicing can be employed
	irrespective of the cycle type of the diagonal.
%
	
	Therefore, based on a similar but simpler
	argument we have

	\begin{proposition}\label{2thm1}
		For any $D$, $ |Y_1|=|Y_2|+|Y_3|$.
	\end{proposition}
	\begin{proof}
		Given $(\mathfrak{p},\epsilon)\in Y_1$ where $\mathfrak{p}=(s,\pi)$. We consider the NTAE
		$\epsilon$ and identify a Case~$2$ transpose $\chi_{h}$,~$h=(i,j,j+1,l)$ as follows:
		assume $\epsilon$ is contained in the cycle
		$$
		C=(s_{i-1} ~v_1^i ~\ldots ~v_{m_i}^i ~s_l ~v_1^l ~\ldots ~v_{m_l}^l ~s_j ~v_1^j ~\ldots ~v_{m_j}^j),
		$$
		where $s_{i-1}=\min\{C\}$, $v_{m_l}^l=\epsilon$,~$s_{j}=\pi(\epsilon)$ and $s_l$ has the
		property that $s_l$ is the smallest in $\{v_1^i,\ldots v_{m_i}^i, s_l,v_1^l,\ldots v_{m_l}^l\}$ such that $s_{j}<_s s_{l}$. Such an element exists by construction
		and we have $s_{i-1}<_s s_j<_s s_l \leq_s \epsilon$.
		
		Let $\mathfrak{p}^h=(s^h,\pi^h)=\chi_h\circ \mathfrak{p}$, we have
		\begin{eqnarray*}
			(s,\pi)=\left(
			\vcenter{\xymatrix@C=0pc@R=1pc{
					\cdots & s_{i-1}\ar@{->>}[d]  & s_i\ar@{--}[dl] &\cdots & s_j\ar@{--}[dl]\ar@{->>}[d] &
					s_{j+1}\ar@{--}[dl] &\cdots & s_{l}\ar@{--}[dl]\ar@{->>}[d] &  \cdots &\epsilon &\cdots \\
					\cdots & \pi(s_{i-1}) & \pi(s_i) & \cdots & \pi(s_j) & \pi(s_{j+1}) &\cdots  &\pi(s_{l}) &\cdots & \pi(\epsilon) &\cdots
				}}\right),\\
				(s^h,\pi^h)=\left(
				\vcenter{\xymatrix@C=0pc@R=1pc{
						\cdots & s_{i-1}\ar@{->>}[d]  & s_{j+1}\ar@{--}[dl] &\cdots & s_l\ar@{--}[dl]\ar@{->>}[d] &
						s_{i}\ar@{--}[dl] &\cdots & s_{j}\ar@{--}[dl]\ar@{->>}[d] &  \cdots &\epsilon &\cdots \\
						\cdots & \pi(s_{j}) & \pi(s_{j+1}) & \cdots & \pi(s_{i-1}) & \pi(s_{i}) &\cdots  &\pi(s_{l}) &\cdots & \pi(\epsilon) &\cdots
					}}
					\right).
				\end{eqnarray*}
				Then,
				$s_{i-1}<_{s^h}s_{l}<_{s^h}s_{j}$. According to Lemma~\ref{2lem4}, $s_{i-1}$,~$s_{j}$,~$s_l$
				will be contained in three distinct
				cycles of $\pi^h$, namely
				$$
				(s_{i-1}  ~v_1^j ~\ldots ~v_{m_j}^j), \quad
				(s_{j} ~v_1^l ~\ldots ~v_{m_l}^l), \quad
				(s_l ~v_1^i ~\ldots ~v_{m_i}^i).
				$$
				It is clear that $s_{{i-1}}$ is still the minimum element w.r.t. $<_{s^h}$ in its cycle.
				By construction we have
				$$
				\{v_1^i, \ldots v_{m_i}^i \}\subset\; ]s_{i-1},s_{j}[\; \cup\; ]s_l,s_{n}] \quad \text{\rm and}
				\quad  \{v_1^l, \ldots v_{m_l}^l\}\subset\; ]s_{i-1},s_j[ \; \cup\; ]s_l,s_{n}]
				$$
				in $s$.
				After transposing $[s_i,s_j]$ and $[s_{j+1},s_l]$, all elements contained in $]s_{i-1},s_{j}[$
				will be larger than $s_l$ in $s^h$ and all elements of $]s_l,s_n]$ remain in $s^h$ to be
				larger than $s_l$. This implies that all elements in the segment $v_1^i  \ldots v_{m_i}^i$
				will be larger than $s_l$ in $s^h$. Accordingly, $s_{l}$ is the minimum element in the cycle
				$(s_l  ~v_1^i ~\ldots ~v_{m_i}^i)$.
				
				It remains to inspect $(s_{j} ~v_1^l ~\ldots ~v_{m_l}^l)$. We find two scenarios:
				\begin{itemize}
					\item[1.] If $s_{j}$ is the minimum (w.r.t.~$<_{s^h}$), then $v_1^l  \ldots v_{m_l}^l$
					contains no element of $]s_{i-1},s_{j}[$ in $s$.
					We claim that in this case there is a bijection between the pairs $(\mathfrak{p},\epsilon)$
					and the set $Y_2$. It suffices to specify the inverse: given an $Y_2$-element,
					$\mathfrak{p'}=(s',\pi')$ with three labeled cycles $(s'_{i-1} ~u_1^i ~\ldots ~u_{m_i}^i)$,
					$(s'_j  ~u_1^j ~\ldots ~u_{m_j}^j)$ and $(s'_l ~u_1^l ~\ldots ~u_{m_l}^l)$ we consider a
					Case~$1$ transpose determined by the three minimum elements, $s'_{i-1}<_{s'} s'_j<_{s'} s'_l$
					in the respective three cycles. This generates a plane permutation $(s,\pi)$
					together with a distinguished NTAE, $\epsilon$, obtained as follows:
					after transposing, the three cycles merge into
					$$
					C=(s'_{i-1} ~u_1^j ~\ldots ~u_{m_j}^j ~s'_j ~u_1^l ~\ldots ~u_{m_l}^l ~s'_l ~u_1^i ~\ldots ~u_{m_i}^i),
					$$
					where $s'_{i-1}<_{s}s'_{l}<_{s}s'_j$. Since elements contained in $u_1^l \ldots u_{m_l}^l$ are
					by construction larger than $s'_{l}$ w.r.t. $<_{s'}$ and these elements will not be moved by
					the transpose, $u_{m_l}^l>_{s} s'_{l}$, i.e., $\epsilon=u_{m_l}^l$ is the NTAE. In case of
					$\{u_1^l, \ldots u_{m_l}^l\}=\varnothing$ we have $\epsilon=s'_j$. The following diagram
					illustrates the situation
					
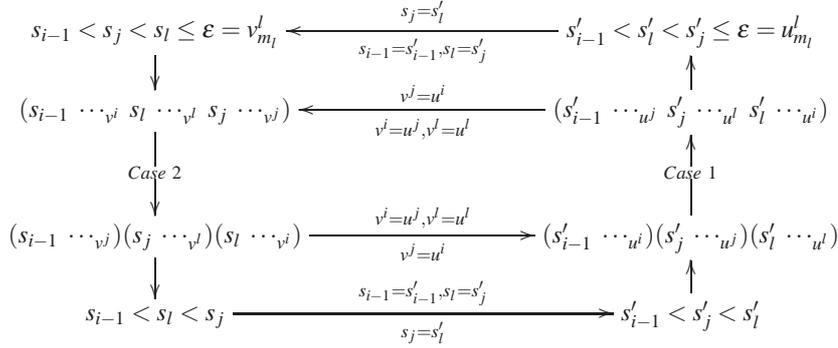
\begin{eqnarray*}
						\xymatrix@R=1pc{
							s_{i-1}<s_j<s_l\leq \epsilon =v_{m_l}^l \ar[d] &&& s'_{i-1}<s'_l<s'_j\leq \epsilon =u_{m_l}^l\ar^{s_{i-1}=s'_{i-1},s_l=s'_j}_{s_j=s'_l}[lll]\\
							(s_{i-1} ~\cdots_{v^i} ~s_l ~\cdots_{v^l} ~s_j ~\cdots_{v^j})\ar[dd]|{Case~ 2} &&& (s'_{i-1} ~\cdots_{u^j} ~s'_j ~\cdots_{u^l} ~s'_l ~\cdots_{u^i})\ar^{v^i=u^j, v^l=u^l}_{v^j=u^i}[lll]\ar[u]\\
							&&&\\
							(s_{i-1} ~\cdots_{v^j})(s_j ~\cdots_{v^l})(s_l ~\cdots_{v^i})\ar[d]\ar^{v^i=u^j, v^l=u^l}_{v^j=u^i}[rrr] &&& (s'_{i-1} ~\cdots_{u^i})(s'_j ~\cdots_{u^j})(s'_l ~\cdots_{u^l})\ar[uu]|{Case~ 1}\\
							s_{i-1}<s_l<s_j \ar^{s_{i-1}=s'_{i-1},s_l=s'_j}_{s_j=s'_l}[rrr] &&&  s'_{i-1}<s'_j<s'_l \ar[u]
						}
					\end{eqnarray*}
					where $\cdots_{v^i}$ denotes the sequence $v_1^i \ldots v_{m_i}^i$.

					\item[2.] If $s_{j}$ is not the minimum, then $\{v_1^l, \ldots v_{m_l}^l \}\neq \varnothing$
					and $\epsilon=v_{m_l}^l$. Since by construction, $\epsilon\in ]s_l,s_n]$ in $s$, it will
					not be impacted by the transposition and we have $s_{j}<_{s^h}\epsilon$. Therefore,
					$\epsilon$ persists to be a NTAE in $\mathfrak{p}^h$.
					We furthermore observe
					$$
					\epsilon>_{s^h}s_{j}>_{s^h}\min\{s_j,v_1^l, \ldots v_{m_l}^l\}>_{s^h}s_l>_{s^h}s_{i-1},
					$$
					where $\min\{s_j,v_1^l, \ldots v_{m_l}^l\}>_{s^h}s_l$ due to the fact that, after transposing $[s_i,s_j]$ and $[s_{j+1},s_l]$, all elements in $\{v_1^l, \ldots v_{m_l}^l\}\subset\; ]s_{i-1},s_j[ \; \cup\; ]s_l,s_{n}]$
					will be larger than $s_l$ following $<_{s^h}$.
					We claim that there is a bijection between such pairs $(\mathfrak{p},\epsilon)$
					and the set $Y_3$. To this end we specify its inverse: given an element in $Y_3$,
					$\mathfrak{p'}=(s',\pi')$ with three labeled cycles
					$$
					(s'_{i-1} ~u_1^i ~\ldots ~u_{m_i}^i), \quad
					(s'_{j}  ~u_1^j ~\ldots ~u_{m_j}^j),\quad
					(s'_l ~u_1^l ~\ldots ~u_{m_l}^l),
					$$
					where $\epsilon=u_{m_l}^l$ is the distinguished NTAE.
					Then a Case~$1$ transpose w.r.t.~the two minima $s'_{i-1}$ and $s'_j$, and $s'_{l}$
					generates a plane permutation, $\mathfrak{p}$, in which $\epsilon$ remains
					as a distinguished NTAE.
				\end{itemize}
				This completes the proof of the proposition.\qed
			\end{proof}
			
			\begin{example}
				Here we look at an example to illustrate the bijection. Consider the plane permutation with $2$ cycles:
				\begin{equation*}
				\mathfrak{p}=\left(\begin{array}{cccccccc}
				3 & 5& 1 & 4 & 8 & 7 & 2& 6\\
				8 & 6& 3 & 5 & 4 & 2 & 7& 1
				\end{array}\right), \quad\text{where}\quad \pi=(3~8~4~5~6~1)(2~7)  .
				\end{equation*}
				Clearly, both $8$ and $6$ are NTAEs. For $(\mathfrak{p},8)$, we find $3,~4,~8$ to determine a Case~$2$ transpose. After the transpose, we obtain
				\begin{equation*}
				\mathfrak{p'}=\left(\begin{array}{cccccccc}
				3 & {\bf 8}&  5& 1 & 4 &  7 & 2& 6\\
				{\bf 5} & 8 & 6& 3 & 4 & 2 & 7& 1
				\end{array}\right), \quad\text{where}\quad \pi'=(3~5~6~1)({\bf 8})(4)(2~7),
				\end{equation*}
				and that $3,~4,~8$ are all the minimum elements in their respective cycles in $\pi'$, i.e., scenario~$1$.
				For the pair $(\mathfrak{p},6)$, we find $3,~1$ and $4$ (the smallest in $\{8,4,5,6\}$ which is larger than $1$) to determine a Case~$2$ transpose. After the transpose, we obtain
				\begin{equation*}
				\mathfrak{p'}=\left(\begin{array}{cccccccc}
				3 & {\bf 4}&  5& 1 & 8 &  7 & 2& 6\\
				{\bf 3} & 8 & 6& 5 & 4 & 2 & 7& 1
				\end{array}\right), \quad\text{where}\quad \pi'=(3)(4~8)(5~{\bf 6}~1)(2~7),
				\end{equation*}
				and that $3,~4$ are the minimum elements in their respective cycles in $\pi'$.
				However, the NTAE $6$ remains as an NTAE. This NTAE needs to be distinguished for the purpose of constructing the reverse map of the bijection. 
				
			\end{example}
			
			Combining Lemma~\ref{2lem6} and Proposition~\ref{2thm1}, we can conclude
			that each plane permutation in $U_D$ with $k$ cycles and a distinguished NTAE
			is in one-to-one correspondence with a plane permutation in $U_D$ having $2i+1$
			labeled cycles among its total $k+2i$ cycles for some $i>0$.

			\begin{theorem}\label{2thm2}
				Let $p_{k}^{\lambda}(n)$ denote the number of $\mathfrak{p}\in U_D$ having
				$k$ cycles where $D$ is of cycle-type $\lambda$.
				Let $p_{a,k}^{\lambda}(n)$ denote the number of $\mathfrak{p}\in U_D$, where
				$\mathfrak{p}$ has $k$ cycles, $Exc(\mathfrak{p})=a$ and $D$ is of type
				$\lambda$. Then,
				\begin{equation}\label{3eq3}
				\sum_{a\geq 0}(n-a-k)p_{a,k}^{\lambda}(n)=
				\sum_{i\geq 1}{k+2i\choose k-1}p_{k+2i}^{\lambda}(n).
				\end{equation}
			\end{theorem}
			\begin{proof}
				Using the notation of Proposition~\ref{2thm1} and recursively applying Lemma~\ref{2lem6}
				as well as Proposition~\ref{2thm1}, we have
				\begin{eqnarray*}
					|Y_1|&=&\sum_{a\geq 0}(n-a-k)p_{a,k}^{\lambda}(n)\\
					&=&|Y_2|+|Y_3|={k+2\choose 3}p_{k+2}^{\lambda}(n)+|Y_3|\\
					&=&{k+2\choose 3}p_{k+2}^{\lambda}(n)+{k+4\choose 5}p_{k+4}^{\lambda}(n)+\cdots
				\end{eqnarray*}
				whence the theorem.\qed
			\end{proof}
			
			{\bf Remark.} Following from Proposition~\ref{2pro1}, the exact number of terms on the RHS of Eq.~\eqref{3eq3} depends on the number of parts in $\lambda$.

	We proceed to study Theorem~\ref{2thm2} in more detail.
	Based on a ``reflection principle"
	argument, we eventually clear the parameter $a$.

	Let $\mu,\eta$ be partitions of $n$. We write $\mu\rhd_{2i+1}\eta$ if $\mu$
	can be obtained by splitting one $\eta$-part into $(2i+1)$ non-zero parts. Let
	furthermore $\kappa_{\mu,\eta}$ denote the number of different ways to obtain
	$\eta$ from $\mu$ by merging $\ell(\mu)-\ell(\eta)+1$ $\mu$-parts into one, where
	$\ell(\mu)$ and $\ell(\eta)$ denote the number of blocks in the partitions $\mu$ and $\eta$,
	respectively.
	
	Let $U_{\lambda}^{\eta}$ denote the set of plane permutations, $\mathfrak{p}=(s,\pi)\in U_D$,
	where $D$ is a fixed permutation of cycle-type $\lambda$ and $\pi$ has cycle-type $\eta$.
	
	\begin{theorem}
		Let $f_{\eta,\lambda}(n)=\vert U_{\lambda}^{\eta} \vert$.
		For $\ell(\eta)+\ell(\lambda)< n+1$, we have
		\begin{equation}
		f_{\eta,\lambda}(n)=\frac{q^{\lambda} \sum_{i\geq 1}\sum_{\mu\rhd_{2i+1}\eta}
			\kappa_{\mu,\eta}f_{\mu,\lambda}(n)
			+q^{\eta} \sum_{i\geq 1}\sum_{\mu\rhd_{2i+1}\lambda}
			\kappa_{\mu,\lambda}f_{\mu,\eta}(n)}{q^{\lambda}[n+1-\ell(\eta)-\ell(\lambda)]}.
		\end{equation}
	\end{theorem}
	
	\begin{proof}
		Let $f_{\eta,\lambda}(n,a)$ denote the number of $\mathfrak{p}\in U_{\lambda}^{\eta}$
		having $a$ exceedances.
		Note that every plane permutation has at least one exceedance. Thus $0\leq a \leq n-1$.
		\begin{claim}
			\begin{equation}\label{2e7}
			\sum_{a\geq 0}(n-a-\ell(\eta))f_{\eta,\lambda}(n,a)=\sum_{i\geq 1}\sum_{\mu\rhd_{2i+1}\eta}
			\kappa_{\mu,\eta}f_{\mu,\lambda}(n).
			\end{equation}
		\end{claim}
		
		Given $\mathfrak{p}=(s,\pi)$ where the cycle-type of $\pi$ is $\eta$,
		a Case~$2$ transpose will result in $\mathfrak{p}^h=(s^h,\pi^h)$ such that $\pi^h$
		has cycle-type $\mu$ and $\mu\rhd_3 \eta$.
		Refining the proof of Proposition~\ref{2thm1}, we observe that each pair $(\mathfrak{p}=(s,\pi),\epsilon)$
		for which $\mathfrak{p}\in U_{\lambda}^{\eta}$ and $\epsilon$ is an NTAE,
		uniquely corresponds to a plane permutation $\mathfrak{p}^h=(s^h,\pi^h) \in U_{\lambda}^{\mu}$
		with $2i+1$ labeled cycles for some $i>0$, and $\mu\rhd_{2i+1}\eta$.
		Conversely, suppose we have $\mathfrak{p}^h=(s^h,\pi^h) \in U_{\lambda}^{\mu}$
		with $\mu\rhd_{2i+1}\eta$. If there are $\kappa_{\mu,\eta}$ ways
		to obtain $\eta$ by merging $2i+1$ $\mu$-parts into one, then we can label $2i+1$ cycles of
		$\mathfrak{p}^h$ in $\kappa_{\mu,\eta}$ different ways, which correspond to $\kappa_{\mu,\eta}$ pairs
		$(\mathfrak{p}=(s,\pi),\epsilon)$ where the cycle-type of $\pi$ is $\eta$ and this implies the Claim.
		
		Immediately, we have
		\begin{equation}\label{2e8}
		\sum_{a\geq 0}(n-a-\ell(\lambda))f_{\lambda,\eta}(n,a)=
		\sum_{i\geq 1}\sum_{\mu\rhd_{2i+1}\lambda}\kappa_{\mu,\lambda}f_{\mu,\eta}(n).
		\end{equation}

		\begin{claim}
			\begin{equation}\label{2e9}
			q^{\lambda} f_{\eta,\lambda}(n,a)=q^{\eta} f_{\lambda,\eta}(n,n-1-a).
			\end{equation}
		\end{claim}
		
		Note that any $\mathfrak{p}=(s,\pi) \in U_{\lambda}^{\eta}$ satisfies $s=D\pi$.
		Taking the inverse to ``reflect" the equation, we uniquely obtain $s^{-1}=\pi^{-1}D^{-1}$.
		The latter can be transformed into an equivalent plane permutation $\mathfrak{p'}=(s',\pi')\in U_{\eta}^{\lambda}$ by conjugation, where elements in $U_{\eta}^{\lambda}$ have a fixed permutation $D'$ of cycle-type $\eta$ as diagonal. Namely, for some $\gamma$, we have
		$$
		s'=\gamma s^{-1} \gamma^{-1}=\gamma (s_0 ~s_{n-1} ~\cdots ~s_1) \gamma^{-1},\quad \pi'=\gamma D^{-1}\gamma^{-1}, \quad
		D_{\mathfrak{p'}}=D'=\gamma \pi^{-1} \gamma^{-1}.
		$$
		Next, we will show that if $\mathfrak{p}$ has $a$ exceedances, the plane permutation $(s^{-1}, D^{-1})$
		has $n-1-a$ exceedances, so that $\mathfrak{p'}$ has $n-1-a$ exceedances according to
		Lemma~\ref{2lem8}. Indeed, if $\mathfrak{p}$ has $a$ exceedances,
		Lemma~\ref{2lem1} guarantees that
		$D$ has $n-(a+1)=n-1-a$ exceedances w.r.t. $<_s$.
		Since an exceedance in $D$ is a strict anti-exceedance (i.e., strictly decreasing) in $D^{-1}$,
		$D^{-1}$ has $n-1-a$ strict anti-exceedances w.r.t. $<_{s}$.
		However, following the linear order $\hat{s}=s_0 s_{n-1} s_{n-2}\cdots s_1$ (induced by $s^{-1}$), any strict anti-exceedance
		w.r.t. $<_{s}$ of $D^{-1}$ the image of which is not $s_0$, will become an exceedance.
		It remains to distinguish the following two situations:
		if $s_0$ is not the image of a strict anti-exceedance,  $s_0$ must be a
		fixed point, so $D^{-1}$ has $n-1-a$ exceedances;
		if $s_0$ is not a fixed point, the strict anti-exceedance having $s_0$ as image remains as a strict
		anti-excceedance in $D^{-1}$. Furthermore, $s_0$ must be an exceedance of
		$D^{-1}$ (w.r.t. $<_s$), and it remains to be an exceedance w.r.t.~$<_{\hat{s}}$.
		In this case, there are also $(n-1-a-1)+1=n-1-a$ exceedances in $D^{-1}$.
		Finally,
		due to the one-to-one correspondence, $f_{\eta,\lambda}(n,a)$ plane permutations $(s,\pi)$ imply that
		$f_{\eta,\lambda}(n,a)$ plane permutations $(s^{-1},D^{-1})$ have $n-1-a$ exceedances.
		Following the same argument as Lemma~\ref{3lem10}, the cardinality of the latter set is also equal to
		$\frac{q^{\eta}}{q^{\lambda}} f_{\lambda,\eta}(n,n-1-a)$, whence
		the claim.
		
		Therefore,
		\begin{align*}
		&\sum_{a\geq 0}(n-a-\ell(\eta)) q^{\lambda} f_{\eta,\lambda}(n,a)+\sum_{a\geq 0}(n-a-\ell(\lambda)) q^{\eta} f_{\lambda,\eta}(n,a)\\
		=& \sum_{a\geq 0}(n-a-\ell(\eta))q^{\lambda} f_{\eta,\lambda}(n,a)+(n-(n-1-a)-\ell(\lambda)) q^{\eta} f_{\lambda,\eta}(n,n-1-a)\\
		=& \sum_{a\geq 0}(n-a-\ell(\eta))q^{\lambda}f_{\eta,\lambda}(n,a)+(n-(n-1-a)-\ell(\lambda))q^{\lambda}f_{\eta,\lambda}(n,a)\\
		=&(n+1-\ell(\eta)-\ell(\lambda))\sum_{a\geq 0} q^{\lambda} f_{\eta,\lambda}(n,a)\\
		=&(n+1-\ell(\eta)-\ell(\lambda)) q^{\lambda} f_{\eta,\lambda}(n).
		\end{align*}
		Multiplying $q^{\lambda}$ and $q^{\eta}$ on both sides of Eq.~\eqref{2e7} and Eq.~\eqref{2e8}, respectively, and
		summing up the LHS and the RHS of Eq.~\eqref{2e7} and Eq.~\eqref{2e8}, respectively, completes the proof. \qed

	\end{proof}

	Summing over all $\eta$ with $\ell(\eta)=k$, we obtain

	\begin{corollary}\label{4cor2}
		For $\ell(\lambda)<n+1-k$, we have
		\begin{equation}\label{3eq111}
		p_k^{\lambda}(n)=\frac{\sum_{i\geq 1}{k+2i\choose k-1}p_{k+2i}^{\lambda}(n) q^{\lambda}
			+\sum_{i\geq 1}\sum_{\mu\rhd_{2i+1}\lambda}
			\kappa_{\mu,\lambda}p_k^{\mu}(n) q^{\mu}}{q^{\lambda} [n+1-k-\ell(\lambda)]}.
		\end{equation}
	\end{corollary}
	\begin{proof}
		For any $\mu$ with $\ell(\mu)=k+2i$, merging any $2i+1$ parts leads to some $\eta$ with $\ell(\eta)=k$ and $\mu\rhd_{2i+1} \eta$.
		Also note, if $\mu\rhd_{2i+1} \eta$ does not hold, $\kappa_{\mu,\eta}=0$.
		Thus, for any $\mu$ with $\ell(\mu)=k+2i$, $\sum_{\eta, \ell(\eta)=k}\kappa_{\mu,\eta}={k+2i\choose 2i+1}$.
		Furthermore, $\sum_{\mu, \ell(\mu)=k+2i}f_{\mu,\lambda}(n)=p_{k+2i}^{\lambda}(n)$.
		Therefore,
		\begin{eqnarray*}
			\sum_{\eta, \atop \ell(\eta)=k}\sum_{i\geq 1}\sum_{\mu\rhd_{2i+1}\eta}
			\kappa_{\mu,\eta}f_{\mu,\lambda}(n) q^{\lambda}&=&
			\sum_{i\geq 1}\sum_{\mu, \atop \ell(\mu)=k+2i}\sum_{\eta, \atop \ell(\eta)=k}
			\kappa_{\mu,\eta}f_{\mu,\lambda}(n) q^{\lambda}\\
			&=& \sum_{i\geq 1}{k+2i\choose k-1}p_{k+2i}^{\lambda}(n) q^{\lambda}.
		\end{eqnarray*}
		We also have
		\begin{eqnarray*}
			\sum_{\eta, \atop \ell(\eta)=k}\sum_{i\geq 1}\sum_{\mu\rhd_{2i+1}\lambda}
			\kappa_{\mu,\lambda}f_{\mu,\eta}(n)q^{\eta}
			&=&\sum_{i\geq 1}\sum_{\mu\rhd_{2i+1}\lambda}
			\kappa_{\mu,\lambda}\sum_{\eta, \atop \ell(\eta)=k}f_{\mu,\eta}(n)q^{\eta}\\
			&=&\sum_{i\geq 1}\sum_{\mu\rhd_{2i+1}\lambda}
			\kappa_{\mu,\lambda}p_k^{\mu}(n) q^{\mu},
		\end{eqnarray*}
		whence the corollary. \qed
	\end{proof}

		Note that $p_1^{\lambda}(n)$ is the number of ways of writing a permutation of cycle-type $\lambda$
		into two $n$-cycles. In Stanley~\cite{stan}, an explicit formula for $p_1^{\lambda}(n)$ was given as,
		if $\lambda=(1^{a_1},2^{a_2},\ldots,n^{a_n})$, then
		\begin{equation*}
		p_1^{\lambda}(n)=\sum_{i=0}^{n-1}\frac{i!(n-1-i)!}{n}\sum_{r_1,\ldots,r_i}
		{a_1-1\choose r_1}{a_2\choose r_2}\cdots{a_i\choose r_i}(-1)^{r_2+r_4+r_6+\cdots},
		\end{equation*}
		where $r_1,\ldots,r_i$ ranges over all non-negative integer solutions of the
		equation $\sum_j jr_j=i$. As a quick application of Corollary~\ref{4cor2},
		we obtain a recurrence for $p_1^{\lambda}(n)$ from which we can obtain simple closed formulas for some particular cases which seems not obvious from Stanley's explicit formula.
		
		\begin{proposition}
			For any $\lambda \vdash n$ and $n-\ell(\lambda)$ even, we have
			\begin{align}
			p_1^{\lambda}(n)=\frac{(n-1)!+\sum_{i \geq 1}\sum_{\mu,\atop \ell(\mu)=2i+\ell(\lambda)}\kappa_{\mu,\lambda}p_1^{\mu}(n)\frac{q^{\mu}}{q^{\lambda}}}{n+1-\ell(\lambda)}.
			\end{align}
			In particular, for $\lambda$ having only small parts, we have
			\begin{align}
			p_1^{1^{a_1}2^{a_2}} &= \frac{(n-1)!}{n+1-a_1-a_2},\\
			p_1^{1^{a_1}2^{a_2}3^1} &= \frac{(n-1)![2(n-a_1-a_2)-3]}{2(n-a_1-a_2)(n-2-a_1-a_2)},\\
			p_1^{1^{a_1}2^{a_2}4^1} &= \frac{(n-1)!(n-1-a_1-a_2)}{(n-2-a_1-a_2)(n-a_1-a_2)}.
			\end{align}
		\end{proposition}
		\proof Setting $k=1$ in eq.~\eqref{3eq111}, we have
		\begin{equation*}
		[n-\ell(\lambda)]	p_1^{\lambda}(n)=\sum_{i\geq 1}{1+2i\choose 0}p_{1+2i}^{\lambda}(n)
		+\sum_{i\geq 1}\sum_{\mu\rhd_{2i+1}\lambda}
		\kappa_{\mu,\lambda}p_1^{\mu}(n) \frac{q^{\mu}}{q^{\lambda}}.
		\end{equation*}
		Note, for $n-\ell(\lambda)$ even, a permutation of cycle-type $\lambda$ can be written as a product of any $n$-cycle and a permutation with $2i+1$ cycles for some $i\geq 0$. Thus, $\sum_{i\geq 0}{1+2i\choose 0}p_{1+2i}^{\lambda}(n)=(n-1)!$ whence the recursion.
		
		For the particular cases, we will only show the second one since the other two follow analogously. For $\lambda=1^{a_1}2^{a_2}3^1$, we observe $\kappa_{\mu,\lambda}\neq 0$ iff $\mu=1^{a_1+3}2^{a_2}$. In this case, $\kappa_{\mu,\lambda}={a_1+3 \choose 3}$.
		Then, using eq.~$(19)$ for $1^{a_1+3}2^{a_2}$, $q^{\lambda}=\frac{n!}{1^{a_1}2^{a_2}3^1 a_1! a_2! 1!}$ and $q^{\mu}=\frac{n!}{1^{a_1+3}2^{a_2} (a_1+3)! a_2!}$, and we obtain the second formula. \qed
	
	}

\section{Another combinatorial proof for Zagier and Stanley's result}

We will provide another combinatorial proof for Zagier and Stanley's result in the following. First, it is obvious that $\xi_{1,k}(n)=0$ if $n-k$ is odd from Proposition~\ref{2pro2}.
In addition, note that $\xi_{1,k}(n)=p_k^{\lambda}(n)$ when $\lambda=n^1$.
Then, from Corollary~\ref{4cor2}, we obtain
	
	\begin{corollary}
		For $k\geq 1$, $n\geq 1$ and $n-k$ is even,
		\begin{equation} \label{3eq2}
		(n+1-k)\xi_{1,k}(n)=\sum_{i\geq 1}{k+2i\choose k-1}\xi_{1,k+2i}(n)+C(n,k).
		\end{equation}
	\end{corollary}
	\proof Inspecting eq.~\eqref{3eq111}, it suffices to show that
	$$
	\sum_{i\geq 1}\sum_{\mu\rhd_{2i+1}\lambda}
	\kappa_{\mu,\lambda}p_k^{\mu}(n) \frac{q^{\mu}}{q^{n^1}}=C(n-k)-\xi_{1,k}(n).
	$$
	To this end, we first observe that for $\lambda=n^1$, $\mu\rhd_{2i+1}\lambda$ iff $\ell(\mu)=2i+1$.
	And $\kappa_{\mu,\lambda}=1$ then.
	Also, by symmetry the number of ways of writing the $n$-cyle $(1~2~\cdots ~n)$ into a product of a permutation with
	$k$ cycles and a permutation of cycle-type $\mu$ equals to $\frac{q^{\mu}}{q^{n^1}} p^{\mu}_k (n)$.
	On the other hand, it is easy to see that ranging over all $\mu \vdash n$, the total number of ways is exactly $C(n,k)$.
	Furthermore, if $n-k$ is even,
	Proposition~\ref{2pro2}
	implies that $(1~2~\cdots ~n)$ can be only factorized into a permutation with $k$ cycles and
	a permutation with $j$ cycles for some odd $j$, i.e., only $\ell(\mu)=2i+1$ matter.
	Thus,
	$$
	\sum_{i\geq 1}\sum_{\mu\rhd_{2i+1}\lambda}
	\kappa_{\mu,\lambda}p_k^{\mu}(n) \frac{q^{\mu}}{q^{n^1}}=\sum_{i\geq 1}\sum_{\mu, \atop \ell(\mu)=2i+1}
	p_k^{\mu}(n) \frac{q^{\mu}}{q^{n^1}}=C(n-k)-\xi_{1,k}(n),
	$$
	completing the proof. \qed
	
	Our idea to prove $\xi_{1,k}(n)=\frac{2}{n(n+1)}C(n+1,k)$ for $n-k$ even is to show both sides satisfy
	the same recurrence and initial conditions.
	To this end, we will relate the obtained results in terms of exceedances of
		plane permutations and exceedances of (ordinary) permutations.
		
		Obviously, exceedances of a plane permutation of the form $(\varepsilon_n,\pi)$
		is the same as exceedances of the ordinary permutation $\pi$.
		Let $\mathfrak{p}=(s,\pi)\in U_D$, where $\mathfrak{p}$ has $a$ exceedances and $k$ cycles.
		Assume $\gamma s\gamma^{-1}=\varepsilon_n=(1~2~\cdots ~n)$. Then, the plane permutation
		$(\varepsilon_n,\gamma \pi \gamma^{-1})$ has $a$ exceedances and $k$ cycles according to Lemma~\ref{2lem8}.
		Furthermore, its diagonal is equal to $\gamma D \gamma^{-1}$ which is of cycle-type
		$\lambda$.
		
		{\bf Observation:} viewing ordinary permutations $\pi$ as plane permutations of the form $(\varepsilon_n,\pi)$
		provides a new way to classify permutations, i.e., by the diagonals.
		
		\begin{lemma}\label{3lem10}
			Let $\hat{p}_{a,k}^{\lambda}(n)$ denote the number of ordinary permutations having $k$ cycles,
			$a$ exceedances and $\lambda$ as the cycle-type of diagonals.
			Let $\hat{p}_{k}^{\lambda}(n)$ denote the number of ordinary permutations having $k$ cycles
			and $\lambda$ as the cycle-type of diagonals.
			Then,
			$$
			q^{\lambda}p_{a,k}^{\lambda}(n)=q^{n^1} \hat{p}_{a,k}^{\lambda}(n)=(n-1)! \hat{p}_{a,k}^{\lambda}(n), \quad
			q^{\lambda}p^{\lambda}_k(n)=(n-1)! \hat{p}^{\lambda}_k (n).
			$$
		\end{lemma}
		\proof Let $S$ be a set of plane permutations $(s,\pi)$ on $[n]$ having $k$ cycles,
		$a$ exceedances and $\lambda$ as the cycle-type of diagonals. Clearly, for any fixed $s$, the number of plane permutations of the form $(s,\pi)$ is the same as the number of plane permutations of the form $(\varepsilon_n,\pi)$ there. Thus,
		$|S|=q^{n^1} \hat{p}_{a,k}^{\lambda}(n)=(n-1)! \hat{p}_{a,k}^{\lambda}(n)$.
		Similarly, the number of plane permutations having a fixed permutation of cycle-type $\lambda$ as diagonal does not depend on specific choice of the permutation. Hence, $|S|=q^{\lambda}p_{a,k}^{\lambda}(n)$, completing the proof of the first equation. The same reasoning leads to the second equation. \qed

		\begin{proposition}
			Let $p_{a,k}(n)$ denote the number of permutations on $[n]$ containing $a$ exceedances and $k$
			cycles. Then,
			\begin{equation}\label{9eq1}
				\sum_{a\geq 0}(n-a-k) p_{a,k}(n)=\sum_{i=1}^{\lfloor\frac{n-k}{2}\rfloor}{k+2i\choose k-1}C(n,k+2i).
			\end{equation}
			In particular, $p_{0,n}(n)=1$, $p_{1,n-1}(n)={n\choose 2}$.
		\end{proposition}
		\begin{proof}
			According to Lemma~\ref{3lem10},
			we have
			\begin{align*}
				p_{a,k}(n)&=\sum_{\lambda} \hat{p}_{a,k}^{\lambda}(n)=\sum_{\lambda} \frac{q^{\lambda}}{(n-1)!}{p}_{a,k}^{\lambda}(n), \\
				C(n,k)&=\sum_{\lambda}  \hat{p}^{\lambda}_k (n)=\sum_{\lambda}\frac{q^{\lambda}}{(n-1)!} p_{k}^{\lambda}(n).
			\end{align*}
			Multiplying $\frac{q^{\lambda}}{(n-1)!}$ on both sides of eq.~\eqref{3eq3} and
			summing over all possible cycle-types $\lambda$
			gives the proposition. \qed
		\end{proof}

		Clearly, we have $\sum_a p_{a,k}(n)=C(n,k)$ and furthermore $\sum_{a}a p_{a,k}(n)$ counts
		the total number of exceedances in all permutations with $k$ cycles. Hence, reformulating Eq.~\eqref{9eq1},
		we have the following corollary:
		\begin{corollary}\label{3cor1}
			The total number of exceedances in all permutations on $[n]$ with $k$ cycles is given by
			\begin{equation}
				\sum_{a}a p_{a,k}(n)=(n-k)C(n,k)-\sum_{i=1}^{\lfloor\frac{n-k}{2}\rfloor}{k+2i\choose k-1}C(n,k+2i).
			\end{equation}
		\end{corollary}
		However, it is easy to compute the total number of exceedances as shown below.
		\begin{proposition}\label{3pro1}
			The total number of exceedances in all permutations on $[n]$ with $k$ cycles is ${n\choose 2}C(n-1,k)$.
		\end{proposition}
		\begin{proof} Note the total number of exceedances in all permutations on $[n]$ with $k$ cycles
			is equal to the size of the set $X$ of permutations $\pi$ on $[n]$ with $k$ cycles and
			with one pair $(i,\pi(i))$ distinguished, where $i$ is an exceedance in $\pi$.
			Let $Y$ denote the set of pairs $(\tau,\alpha)$, where $\tau$ is a subset of $[n]$ having $2$ elements
			and $\alpha$ is a permutation on $[n-1]$ having $k$ cycles. We will show that there is a bijection between
			$X$ and $Y$.
			Given $(\pi, (i,\pi(i)))\in X$, we obtain $(\tau,\alpha)\in Y$ as follows: set $\tau=\{i,\pi(i)\}$ and
			$\alpha'$ on $[n]\setminus\{\pi(i)\}$ as $\alpha'(j)=\pi(j)$ if $j\neq i$ while $\alpha'(i)=\pi^2(i)$.
			Now we obtain $\alpha$ from $\alpha'$ by substituting $x-1$ for every number $x>\pi(i)$.
			Conversely, given $(\tau,\alpha)\in Y$, where $\tau=\{a,b\}$ and $a<b$. Define $\alpha'$ from $\alpha$
			by substituting $x+1$ for every number $x\geq b$. Next we define $\pi$ from $\alpha'$ in
			the following way:
			$\pi(j)=\alpha'(j)$ if $j\neq a, b$ while $\pi(a)=b$ and $\pi(b)=\alpha'(a)$.
			Note that by construction $a$ is an exceedance in $\pi$ and clearly, $|Y|={n\choose 2}C(n-1,k)$,
			whence the proposition. \qed
		\end{proof}

		Corollary~\ref{3cor1} and Proposition~\ref{3pro1} give rise to a new recurrence for the
		unsigned Stirling numbers of the first kind $C(n,k)$.
		\begin{theorem}\label{3thm4} For $n\geq 1, k\geq 1$, we have
			\begin{equation}\label{3eq4}
				C(n+1,k)=\sum_{i\geq 1}{k+2i\choose k-1}\frac{C(n+1,k+2i)}{n+1-k}+{n+1\choose 2}\frac{C(n,k)}{n+1-k}.
			\end{equation}
		\end{theorem}

		Reformulating Eq.~\eqref{3eq4}, we obtain
		\begin{equation}\label{3eq5}
			\frac{2C(n+1,k)}{n(n+1)}=\sum_{i\geq 1}{k+2i\choose k-1}\frac{1}{n+1-k}\frac{2C(n+1,k+2i)}{n(n+1)}+\frac{C(n,k)}{n+1-k}.
		\end{equation}

		Comparing Eq.~\eqref{3eq2} and Eq.~\eqref{3eq5}, we observe that $\frac{2}{n(n+1)}C(n+1,k)$
		and $\xi_{1,k}(n)$ satisfy the same recurrence. Furthermore, the
		initial value $\xi_{1,n}(n)$ is equal to the number of different ways to factorize
		an $n$-cycle into an $n$-cycle and a permutation with $n$ cycles. Since only the identity map
		has $n$ cycles, we have $\xi_{1,n}(n)=1$. On the other hand, $C(n+1,n)$ is the number of
		permutations on $[n+1]$ with $n$ cycles. Such permutations have cycle-type $1^{n-1}2^1$.
		It suffices to determine the $2$-cycle, which is equivalent to selecting $2$ elements from $[n+1]$.
		Therefore, the initial value $\frac{2}{n(n+1)}C(n+1,n)=\frac{2}{n(n+1)}{n+1\choose 2}=1$.
		Thus, $\frac{2}{n(n+1)}C(n+1,k)$ and $\xi_{1,k}(n)$ agree on the initial values.
		So, we have
		
		\begin{proposition}[Zagier~\cite{zag}, Stanley~\cite{stan3}]\label{3cor3}
			For $k\geq 1$, $n\geq 1$ and $n-k$ even, we have
			\begin{equation}
				\xi_{1,k}(n)=\frac{2}{n(n+1)}C(n+1,k).
			\end{equation}
		\end{proposition}

\section{Transposition distance of permutations}

{
In Bioinformatics, comparative study of genome sequences is a very important tool to understand evolution.
In particular, the problem of determining the minimum number of certain operations required to transform
one of two given genome sequences into the other, is extensively studied.
Combinatorially, this problem can be formulated as sorting a given permutation (or sequence)
to the identity permutation by certain priori prescribed operations, in a minimum number of steps.
This minimum number is called the distance of the permutation to be sorted w.r.t. the operations chosen.
Common operations studied are transpositions~\cite{bul,pev1,chri1,eli,labarre2},
block-interchanges~\cite{bona,chri1,chri2,hhtl,lin} and
reversals~\cite{yan,braga,pev2,capr,pev4,pev3}, etc.

There were two main existing approaches to study these distance problems:
one is based on graph models, e.g., cycle-graphs and breakpoint graphs~\cite{pev1,pev2,pev3,chri1}, where properties of these graphs (e.g., cycle decomposition) were used to characterize the distances;
the other  is based on permutation group theory, see for instance~\cite{feimei,hhtl,huanglu,labarre2,Meid}, where operations (e.g., transpositions) were modeled as short cycles (e.g., $2$-cycles, $3$-cycles) acting on a permutation induced by the sequence to be sorted so that the distance of the sequence can be obtained by computing the number of short cycles needed from the permutation group theory.

In the rest of sections, we will study the transposition distance and block-interchange distance for
permutations, as well as the reversal distance for signed permutations using a unified plane permutation framework.
The motivation comes from the following observation:
for a given one-line permutation $s$, if we view it as a part of a cycle $\bar{s}$ coming from a plane permutation
$(\bar{s},\gamma)$, then any swap of two segments in $s$ induces a transposition action on the plane permutation $(\bar{s}, \gamma)$. 
Moreover, we have the freedom of choosing $\gamma$,
which allows us to study distance problems by solving optimization problems instead of relying on certain ad hoc constructions
(like cycle-graphs and breakpoint graphs).

Along this line, we will obtain general lower bounds for the transposition distance and the block-interchange distance. Comparing with existing lower bounds (which will be made explicit later), it turns out that these existing lower bounds are equivalent to evaluations
at a particular $\gamma$. 
}

Let us start to look at the transposition distance of permutations. Given a sequence (one-line permutation) on $[n]$
$$
s=a_1\cdots a_{i-1}a_i\cdots a_ja_{j+1}\cdots a_k a_{k+1}\cdots a_{n},
$$
a transposition action on $s$ means to change $s$ into
$$
s'=a_1\cdots a_{i-1}a_{j+1}\cdots a_k a_i\cdots a_j a_{k+1}\cdots a_{n},
$$
for some $1\leq i\leq j<k\leq n$. Let $e_n=123\cdots n$.
The transposition distance of a sequence $s$ on $[n]$ is the minimum number of transpositions
needed to sort $s$ into $e_n$. Denote this distance as $td(s)$.

Let $C(\pi)$,~$C_{odd}(\pi)$ and $C_{ev}(\pi)$ denote the number of cycles, the number of odd cycles
and the number of even cycles in $\pi$, respectively. Furthermore, let $[n]^*=\{0,1,\ldots, n\}$, and
$$
\hat{e}_n=(0~1~2~3~\cdots ~n), \quad \bar{s}=(0 ~a_1 ~a_2 ~\cdots ~a_n), \quad p_t=(n ~n-1 ~\ldots ~1 ~0).
$$

\begin{theorem}\label{5thm1}
\begin{eqnarray}
td(s)\geq \max_{\gamma}\left\{\frac{\max\{|C(p_t\bar{s}\gamma)-C(\gamma)|,|C_{odd}(p_t\bar{s}\gamma)-C_{odd}(\gamma)|,
|C_{ev}(p_t\bar{s}\gamma)-C_{ev}(\gamma)|\}}{2}\right\},
\end{eqnarray}
where $\gamma$ ranges over all permutations on $[n]^*$.
\end{theorem}

\begin{proof} 
For an arbitrary permutation $\gamma$ on $[n]^*$, $\mathfrak{p}=(\bar{s},\gamma)$ is a plane permutation.
By construction, each transposition on the sequence $s$ induces a transpose on $\mathfrak{p}$.
(The auxiliary element $0$ is used to handle the case where $a_1$ is contained in the first
block of the transpositions, because the bottom row of a transposition action on a plane permutation
is forward shifted.)
If $s$ changes to $e_n$ by a series of transpositions, we have, for some $\beta$, that $\mathfrak{p}$ changes into the
plane permutation $(\hat{e}_n, \beta)$.
By construction, we have
$$
D_{\mathfrak{p}}=\bar{s}\gamma^{-1}=\hat{e}_n\beta^{-1},
$$
and accordingly
$$
\beta=\gamma {\bar{s}}^{-1}\hat{e}_n.
$$
Since each transpose changes the number of cycles
by at most $2$ according to Lemma~\ref{2lem3},
at least $\frac{|C(\gamma {\bar{s}}^{-1}\hat{e}_n)-C(\gamma)|}{2}=
\frac{|C(p_t\bar{s}\gamma^{-1})-C(\gamma^{-1})|}{2}$
transposes are needed from $\gamma$ to $\beta$.
The same argument also applies to deriving the lower bounds in terms of odd and even cycles, respectively.
Note that $\gamma$ can be arbitrarily selected, then the proof follows.  \qed
\end{proof}

In this general formulation of Theorem~\ref{5thm1}, 
setting $\gamma=(p_t\bar{s})^{-1}$ we immediately obtain
\begin{corollary}\label{cor-bp}
\begin{eqnarray}
td(s)& \geq & \frac{n+1-C(p_t\bar{s})}{2},\\
td(s)& \geq & \frac{n+1-C_{odd}(p_t\bar{s})}{2}
\end{eqnarray}
\end{corollary}

The most common graph model used to study transposition distance is cycle-graph proposed by Bafna and Pevzner~\cite{pev1}.
Given a permutation $s=s_1s_2\cdots s_n$ on $[n]$, the cycle graph $G(s)$ of $s$ is obtained
as follows:
add two additional elements $s_0=0$ and $s_{n+1}=n+1$. The vertices of $G(s)$ are the elements in $[n+1]^*$.
Draw a directed black edge from $i+1$ to $i$, and draw a directed gray edge from
$s_i$ to $s_{i+1}$, we then obtain $G(s)$.
An alternating cycle in $G(s)$ is a directed cycle, where its edges alternate in color.
An alternating cycle is called odd if the number of black edges is odd.
Bafna and Pevzner obtained lower and upper bound for $td(s)$ in terms of
the number of cycles and odd cycles of $G(s)$~\cite{pev1}.

By examining the cycle graph model $G(s)$ of a permutation $s$, it turns out
the cycle graph $G(s)$ is actually the directed graph representation of the
product $\bar{s}^{-1}p_t^{-1}$, if we identify the two auxiliary points $0$ and $n+1$.
The directed graph representation of a permutation $\pi$ is the directed graph
by drawing an directed edge from $i$ to $\pi(i)$. If we color the directed edge
of $\bar{s}^{-1}$ gray and the directed edge of $p_t^{-1}$ black, an alternating cycle
then determines a cycle of the permutation $\bar{s}^{-1}p_t^{-1}$ (thus $p_t\bar{s}$). Therefore, the number of
cycles and odd cycles in $p_t\bar{s}$ is equal to the number of
cycles and odd cycles in $G(s)$, respectively.
As a result, the lower bounds in Corollary~\ref{cor-bp} are exactly the same as the lower bounds
obtained by Bafna and Pevzner~\cite{pev1} and this relation was also derived in~\cite{labarre1,labarre2}.
In particular, in~\cite{labarre2}, this lower bound was obtained using permutations and by translating the
transposition distance of $s$ into the minimum number of $3$-cycles, $p_t\bar{s}$ can be factored into.

In view of Theorem~\ref{5thm1} we next ask: is it possible by employing an appropriate $\gamma$, to improve the lower
bounds of in Corollary~\ref{cor-bp}?
I.e.~given a permutation $\pi$, what is the maximum number of
$|C(\pi\gamma)-C(\gamma)|$ (resp. $|C_{odd}(\pi\gamma)-C_{odd}(\gamma)|$, $|C_{ev}(\pi\gamma)-C_{ev}(\gamma)|$),
where $\gamma$ ranges over a set of permutations.

More generally, we can study the distribution functions
\begin{equation}
\sum_{\gamma\in A} z^{C(\pi\gamma)-C(\gamma)},\quad \sum_{\gamma\in A} z^{C_{odd}(\pi\gamma)-C_{odd}(\gamma)},
\quad \sum_{\gamma\in A} z^{C_{ev}(\pi\gamma)-C_{ev}(\gamma)},
\end{equation}
where $A$ is a set of permutations, e.g., a conjugacy class or all permutations.
In this paper, we will later determine $\max_{\gamma}\{|C(\pi\gamma)-C(\gamma)|\}$ for
an arbitrary permutation $\pi$. Surprisingly, the maximum for this case is achieved when
$\gamma=\pi^{-1}$ or $\gamma$ is the identity permutation.
For the other two problems in terms of odd cycle and
even cycle, we are unable to solve it at present.
{
However,
the following example shows that for even cycles, the maximum is not necessarily achieved by
$\gamma=\pi^{-1}$ or $\gamma=$identity.
\begin{example}
	Suppose $p_t=(4,3,2,1,0)$ and $\bar{s}=(0,1,2,4,3)$. Consider $|C_{ev}(p_t \bar{s}\gamma)-C_{ev}(\gamma)|$.
	\begin{itemize}
		\item For $\gamma=(p_t \bar{s})^{-1}$ or $\gamma=$identity, $p_t \bar{s}=(0)(1)(2,3,4)$ so that $|C_{ev}(p_t \bar{s}\gamma)-C_{ev}(\gamma)|=0$;
		\item For $\gamma=(0)(1,3)(2,4)$, $p_t \bar{s}\gamma=(0)(1,4,3)(2)$ so that $|C_{ev}(p_t \bar{s}\gamma)-C_{ev}(\gamma)|=2$.
	\end{itemize}
\end{example} 
}

Another approach to obtain a better lower bound is fixing $\gamma$ and figuring out these unavoidable
transposes which do not increase (or decrease) the number of cycles (or odd, or even cycles) from
$\gamma$ to $\gamma {\bar{s}}^{-1}\hat{e}_n$. In particular, by setting $\gamma=p_t \bar{s}$,
it is not hard to analyze the number of ``hurdles" similar as in Christie~\cite{chri1} in the framework
of plane permutations,
which we do not go into detail here.

\section{Block-interchange distance of permutations}

A more general transposition problem, where the involved two blocks are not necessarily adjacent,
was studied in Christie~\cite{chri2}. It is referred to as the block-interchange distance problem.
The minimum number of block-interchanges needed to sort $s$ into $e_n$ is accordingly called the
block-interchange distance of $s$ and denoted as $bid(s)$.

{
Following Lemma~\ref{2lem5} and the same reasoning as in the proof of Theorem~\ref{5thm1}, we immediately obtain
\begin{align}\label{gen-bid}
bid(s)\geq \frac{\max_{\gamma}\{|C(p_t\bar{s} \gamma)-C(\gamma)|\}}{2},
\end{align}
where $\gamma$ ranges over all permutations on $[n]^*$.
Christie~\cite{chri2} proved an exact formula for the block-interchange distance which implies that the maximum of the RHS of eq.~\eqref{gen-bid} is achieved by $\gamma=(p_t \bar{s})^{-1}$. For completeness, we give a simple proof of this fact here.
}
\begin{lemma}\label{5lem1}
Let $\mathfrak{p}=(\bar{s},\pi)$ be a plane permutation on $[n]^*$
where $D_{\mathfrak{p}}=p_t^{-1}$ and $\bar{s}\neq \hat{e}_n$.
Then, there exist $\bar{s}_{i-1}<_{\bar{s}} \bar{s}_j <_{\bar{s}} \bar{s}_{k-1} \leq_{\bar{s}} \bar{s}_l$ such that
$$
\pi(\bar{s}_{i-1})=\bar{s}_{k-1}, \quad \pi(\bar{s}_l)=\bar{s}_j.
$$
\end{lemma}
\begin{proof} Since $\bar{s}\neq \hat{e}_n$, there exists $x\in [n]$ such that $x+1<_{\bar{s}} x$.
Assume $x=\bar{s}_{k-1}$ is the largest such integer and let $\bar{s}_{i}=x+1$.
Then, $\pi(\bar{s}_{i-1})=x=\bar{s}_{k-1}$ since $D_{\mathfrak{p}}(\pi(\bar{s}_{i-1}))
=\pi(\bar{s}_{i-1})+1=x+1$.
Between $\bar{s}_{i-1}$ and $x$, find the largest integer which is larger than $x$.
Since $x+1$ lies between $\bar{s}_{i-1}$ and $x$, this maximum exists and we denote it by $y$. Then we
have by construction
$$
\bar{s}_{i-1}<_{\bar{s}} \bar{s}_j=y <_{\bar{s}} x= \bar{s}_{k-1} <_{\bar{s}} y+1=\bar{s}_{l+1}.
$$
Therefore, $\pi(\bar{s}_{l})=D_{\mathfrak{p}}^{-1}(y+1)=y=\bar{s}_j$, whence the lemma. \qed
\end{proof}

Then, we obtain
\begin{theorem}[Christie~\cite{chri2}]\label{5thm2}
\begin{equation}
bid(s)= \frac{n+1-C(p_t\bar{s})}{2}.
\end{equation}
\end{theorem}
\begin{proof}
Let $\mathfrak{p}=(\bar{s},\pi)$ be a plane permutation on $[n]^*$
where $D_{\mathfrak{p}}=p_t^{-1}$ and $\bar{s}\neq \hat{e}_n$.
According to Lemma~\ref{5lem1}, we either have ${\bar{s}}_{i-1}<_{\bar{s}} {\bar{s}}_j <_{\bar{s}} {\bar{s}}_{k-1} <_{\bar{s}} {\bar{s}}_l$
such that we either have $\pi$-cycle
$$
({\bar{s}}_{i-1} ~{\bar{s}}_{k-1} ~\ldots ~{\bar{s}}_l ~{\bar{s}}_j ~\ldots)\quad \mbox{or} \quad ({\bar{s}}_{i-1} ~{\bar{s}}_{k-1} ~\ldots)({\bar{s}}_l ~{\bar{s}}_j ~\ldots),
$$
or ${\bar{s}}_{i-1}<_{\bar{s}} {\bar{s}}_j <_{\bar{s}} {\bar{s}}_{k-1} =_{\bar{s}} {\bar{s}}_l$ such that
we have the $\pi$-cycle $({\bar{s}}_{i-1} ~{\bar{s}}_{k-1} ~{\bar{s}}_j ~\ldots)$.
For the former case, the determined $\chi_h$ is either Case~$c$ or Case~$e$ of Lemma~\ref{2lem5}.
For the latter case, the determined $\chi_h$ is Case~$2$ transpose of Lemma~\ref{2lem4}.
Therefore, no matter which case, we can always find a block-interchange to
increase the number of cycles by $2$. Then,  arguing as in Theorem~\ref{5thm1} completes
the proof.\qed
\end{proof}

{
Theorem~\ref{5thm2} was proved in~\cite{lin} using permutations
by translating the block-interchange distance of $s$ into the minimum number of 
pairs of $2$-cycles the permutation $p_t \bar{s}$ can be factored into.
}

Furthermore, Zagier and Stanley's result mentioned earlier implies that

\begin{corollary}\label{5cor7}
Let $bid_k(n)$ denote he number of sequences $s$ on $[n]$ such that $bid(s)=k$.
Then,
\begin{equation}
bid_k(n)=\frac{2C(n+2,n+1-2k)}{(n+1)(n+2)}.
\end{equation}
\end{corollary}
\begin{proof}
Let
$$
k=bid(s)= \frac{n+1-C(p_t\bar{s})}{2}.
$$
The number of $s$ such that $bid(s)=k$ is equal to the number of permutation $\bar{s}$
such that $C(p_t \bar{s})=n+1-2k$. Then, applying Zagier and Stanley's result
completes the proof. \qed
\end{proof}

We note that the corollary above was also used by Bona and Flynn~\cite{bona} to compute the average number of
block-interchanges needed to sort permutations.

In view of the general lower bound for block-interchanges eq.~\eqref{gen-bid} and Theorem~\ref{5thm2}, we are now in position to answer
one of the optimization problems mentioned earlier.

\begin{theorem}\label{5thm3}
Let $\alpha$ be a permutation on $[n]$ and $n\geq 1$. Then we have
\begin{equation}
\max_{\gamma}\{|C(\alpha \gamma)-C(\gamma)|\}=n-C(\alpha),
\end{equation}
where $\gamma$ ranges over all permutations on $[n]$.
\end{theorem}

\begin{proof}
	First, from eq.~\eqref{gen-bid} and Theorem~\ref{5thm2}, we have:
for arbitrary $s$,
\begin{equation}
\max_\gamma\{|C(p_t\bar{s} \gamma)-C(\gamma)|\}=n+1-C(p_t\bar{s}),
\end{equation}
where $\gamma$ ranges over all permutations on $[n]^*$.

We now use the fact that any even permutation $\alpha'$ on $[n]^*$ has a factorization into
two $(n+1)$-cycles. Assume $\alpha'=\beta_1\beta_2$ where $\beta_1, \beta_2$ are two
$(n+1)$-cycles, and $p_t=\theta \beta_1 \theta^{-1}$.
Then, we have
\begin{eqnarray*}
\max_{\gamma}\{|C(\alpha' \gamma)-C(\gamma)|\}&=&\max_{\gamma}\{|C(\theta \beta_1\beta_2\gamma\theta^{-1})-C(\theta\gamma\theta^{-1})|\}\\
&=&\max_{\gamma}\{|C(p_t\theta\beta_2\theta^{-1}\theta\gamma\theta^{-1})-C(\theta\gamma\theta^{-1})|\}\\
&=& n+1-C(p_t\theta\beta_2\theta^{-1})\\
&=& n+1- C(\beta_1\beta_2)=n+1-C(\alpha').
\end{eqnarray*}
So the theorem holds for even permutations.
Next we assume that $\alpha$ is an odd permutation. If $C(\alpha)<n$, then we can always find a transposition $\tau$ (i.e., a cycle of length $2$)
such that $\alpha= \alpha' \tau$, where $\alpha'$ is an even permutation and $C(\alpha)=C(\alpha')-1$. Thus,
\begin{align*}
\max_{\gamma}\{|C(\alpha \gamma)-C(\gamma)|\}&= \max_{\gamma}\{|C(\alpha' \tau\gamma)-C(\gamma)|\}\\
&=\max_{\gamma}\{|C(\alpha' \tau\gamma)-C(\tau\gamma)+C(\tau\gamma)-C(\gamma)|\}\\
& \leq \max_{\gamma}\{|C(\alpha' \tau\gamma)-C(\tau\gamma)|+|C(\tau\gamma)-C(\gamma)|\}\\
&=[n-C(\alpha')]+1=n-C(\alpha).
\end{align*}
Note that $|C(\alpha I)-C(I)|=n-C(\alpha)$, where $I$ is the the identity permutation.
Hence, we conclude that $\max_{\gamma}\{|C(\alpha \gamma)-C(\gamma)|\}=n-C(\alpha)$.
When $C(\alpha)=n$, i.e., $\alpha=I$, it is obvious that $\max_{\gamma}\{|C(I \gamma)-C(\gamma)|\}=0=n-C(\alpha)$.
Hence, the theorem holds for odd permutations as well, completing the proof.\qed
\end{proof}

\section{Reversal distance for signed permutations}
In this section, we consider the reversal distance for signed permutations, a problem
extensively studied in the context of genome evolution~\cite{pev2,pev3,yan} and the references therein.
Lower bounds for the reversal
distance based on the breakpoint graph model were obtained in~\cite{pev2,pev3,pev4}.

In our framework the reversal distance problem can be expressed as a
block-interchange distance problem. A lower bound can be easily obtained in
this point of view, and the lower bound will be shown to be the exact reversal distance
for most of signed permutations.

Let $[n]^{-}=\{-1,-2,\ldots, -n\}$.
\begin{definition}
A signed permutation on $[n]$ is a pair $(a,w)$ where $a$ is
a sequence on $[n]$ while $w$ is a word of length $n$ on the
alphabet set $\{+,-\}$.
\end{definition}
Usually, a signed permutation is represented by a single sequence
$a_w=a_{w,1}a_{w,2}\cdots a_{w,n}$ where $a_{w,k}=w_ka_k$, i.e.,
each $a_k$ carries a sign determined by $w_k$.

Given a signed permutation $a=a_1a_2\cdots a_{i-1}a_i a_{i+1}\cdots a_{j-1}
a_j a_{j+1}\cdots a_n$ on $[n]$, a reversal $\varrho_{i,j}$ acting on
$a$ will change $a$ into
$$
a'=\varrho_{i,j}\diamond a =a_1a_2\cdots a_{i-1}(-a_j)(-a_{j-1})\cdots
(-a_{i+1})(-a_i)a_{j+1}\cdots a_n.
$$
The reversal distance $d_r(a)$ of a signed permutation $a$ on $[n]$ is the minimum
number of reversals needed to sort $a$ into $e_n=12\cdots n$.

For the given signed permutation $a$, we associate the sequence $s=s(a)$ as follows
$$
s=s_0s_1s_2\cdots s_{2n}=0a_1a_2\cdots a_n (-a_n)(-a_{n-1})\cdots (-a_2)(-a_1),
$$
i.e., $s_0=0$ and $s_{k}=-s_{2n+1-k}$ for $1\leq k\leq 2n$. Furthermore, such sequences
will be referred to as skew-symmetric sequences since we have $s_{k}=-s_{2n+1-k}$. A sequence
$s$ is called exact if there exists $s_i<0$ for some $1\le i\le n$.
The reversal distance of $a$ is equal to the block-interchange distance of $s$
into
$$
e_n^{\natural}=012\cdots n (-n)(-n+1)\cdots (-2) (-1),
$$
where only certain block-interchanges are allowed, i.e., only the actions $\chi_h$,
$h=(i,j,2n+1-j,2n+1-i)$ are allowed where $1\leq i\leq j\leq n$. Hereafter, we will denote
these particular block-interchanges on $s$ as reversals, $\varrho_{i,j}$.

Let
\begin{eqnarray*}
\tilde{s}&=&(s)=(0 ~a_1 ~a_2 ~\ldots ~a_{n-1} ~a_n ~-a_n ~-a_{n-1} ~\ldots ~-a_2 ~-a_1),\\
p_r&=&(-1 \quad -2 \quad\cdots \quad -n+1 \quad -n \quad n \quad n-1 \quad\cdots \quad 2 \quad 1 \quad 0).
\end{eqnarray*}
A plane permutation of the form $(\tilde{s},\pi)$ will be called skew-symmetric.

\begin{theorem}\label{7thm1}
\begin{equation}\label{5eq1}
d_r(a)\geq \frac{2n+1-C(p_r\tilde{s})}{2}.
\end{equation}
\end{theorem}
\begin{proof} Since reversals are restricted block-interchanges, the reversal distance will be
bounded by the block-interchange distance without restriction. Theorem~\ref{5thm2} then implies Eq.~\eqref{5eq1}. \qed
\end{proof}

Our approach gives rise to the question of how potent the restricted block-interchanges are.
Is it difficult to find a block-interchange increasing the number of cycles by $2$ that is a
reversal (i.e., $2$-reversal)?

We will call a plane permutation $(\tilde{s},\pi)$ exact, skew-symmetric if
$\tilde{s}$ is exact and skew-symmetric. The following lemma will show that there is almost
always a $2$-reversal.

\begin{lemma}\label{7lem1}
Let $\mathfrak{p}=(\tilde{s},\pi)$ be exact and skew-symmetric on $[n]^*\cup[n]^{-}$,
where $D_{\mathfrak{p}}=p_r^{-1}$. Then, there always exist $i-1$ and $2n-j$ such that
\begin{equation}
\pi({s}_{i-1})={s}_{2n-j},
\end{equation}
where $0\leq i-1\leq n-1$ and $n+1\leq 2n-j \leq 2n$. Furthermore,
we have the following cases
\begin{itemize}
\item[(a)] If $s_{i-1}<_s s_{j}<_s s_{2n-j} <_s s_{2n+1-i}$, then
\begin{equation}
\pi({s}_{i-1})={s}_{2n-j}, \quad \pi(s_{j})=s_{2n+1-i}.
\end{equation}
\item[(b)] If $s_{j}<_s s_{i-1}<_s s_{2n+1-i}<_s s_{2n-j} $, then
\begin{equation}
\pi({s}_{i-1})={s}_{2n-j},\quad \pi(s_{j})=s_{2n+1-i}.
\end{equation}
\end{itemize}
\end{lemma}
\begin{proof} We firstly prove the former part.
Assume $s_{i}$ is the smallest negative element among the subsequence $s_1s_2\cdots s_n$.
If $s_{i}=-n$, then we have $s_{2n+1-i}=-s_i=n$ by symmetry. Since $D_{\mathfrak{p}}=p_r^{-1}$,
for any $k$, $s_{k+1}=p_r^{-1}(s_{k})=s_k+1$ where $n+1$ is interpreted as $-n$. Thus,
$\pi(s_{i-1})=D_{\mathfrak{p}}^{-1}(s_i)=D_{\mathfrak{p}}^{-1}(-n)=n=s_{2n+1-i}$.
Let $2n-j=2n+1-i$, then $2n-j\geq n+1$ and we are done.
If $s_{i}>-n$, then we have $\pi(s_{i-1})=D_{\mathfrak{p}}^{-1}(s_i)=s_{i}-1\geq -n$.
Since $s_{i}$ is the smallest negative element among $s_t$ for $1\leq t\leq n$,
if $s_{2n-j}=s_{i}-1< s_i$, then $2n-j\geq n+1$, whence the former part.

Using $D_{\mathfrak{p}}=p_r^{-1}$ and the skew-symmetry $s_k=-s_{2n+1-k}$,
we have in case of (a) the following situation in $\mathfrak{p}$ (only relevant entries are illustrated)
\begin{eqnarray*}
\left(\begin{array}{cccccccccccc}
i-1 & i &\cdots & j & j+1 & \cdots & 2n-j & 2n+1-j & \cdots & 2n+1-i \\\hline
s_{i-1}&(s_{2n-j}+1)&\cdots & s_{j}&-s_{2n-j}&\cdots & s_{2n-j}&-s_{j}& \cdots & (-s_{2n-j}-1) \\
s_{2n-j}&\diamondsuit & \cdots & (-s_{2n-j}-1)&\diamondsuit &\cdots &\diamondsuit&\diamondsuit &\cdots & \diamondsuit
\end{array}\right).
\end{eqnarray*}
Therefore, we have
\begin{eqnarray*}
\pi(s_{i-1})=s_{2n-j},\quad \pi(s_{j})=-s_{2n-j}-1=s_{2n+1-i}.
\end{eqnarray*}
Analogously we have in case of $(b)$ the situation
\begin{eqnarray*}
\left(\begin{array}{cccccccccccc}
j & j+1 & \cdots & i-1 & i &\cdots &  2n+1-i & 2n+2-i & \cdots & 2n-j \\\hline
s_{j}& -s_{2n-j}&\cdots & s_ {i-1}& s_{2n-j}+1 &\cdots & -s_{2n-j}-1 &-s_{i-1}& \cdots & s_{2n-j} \\
-s_{2n-j}-1 &\diamondsuit & \cdots & s_{2n-j}&\diamondsuit &\cdots &\diamondsuit&\diamondsuit &\cdots & \diamondsuit
\end{array}\right).
\end{eqnarray*}
Therefore, we have
\begin{eqnarray*}
\pi(s_{{i-1}})=s_{2n-j},\quad \pi(s_{j})=-s_{2n-j}-1=s_{2n+1-i}.
\end{eqnarray*}
This completes the proof. \qed
\end{proof}

{\bf Remark.} The pair $s_{i-1}$ and $s_{2n-j}$ such that $\pi({s}_{i-1})={s}_{2n-j}$ is not unique.
For instance, assume the positive integer $k$, $1 \leq k \leq n-1$, is not in the subsequence $s_1s_2
\cdots s_n$ but $k+1$ is, then $\pi^{-1}(k)$ and $k=D_{\mathfrak{p}}^{-1}(k+1)$ form such a pair.


Inspection of Lemma~\ref{2lem5} and Lemma~\ref{7lem1} shows that
there is almost always a $2$-reversal for signed permutations.
The only critical cases, not covered in Lemma~\ref{7lem1}, are
\begin{itemize}
\item The signs of all elements in the given signed permutation are positive.
\item Exact signed permutation which for $1\leq i \leq n$ and $ n+1\leq 2n-j$,
$\pi(s_{i-1})=s_{2n-j}$ iff $2n-j=2n+1-i$.
\end{itemize}
We proceed to analyze the latter case. Since $\pi(s_{i-1})=s_{2n+1-i}=-s_i$, we have
\begin{eqnarray*}
\left(\begin{array}{cccccccc}
s_{i-1}&s_{i}&\cdots& s_n&-s_n&\cdots &s_{2n+1-i}&s_{2n+2-i}\\
\pi(s_{i-1})&\diamondsuit&\cdots &\diamondsuit&\diamondsuit&\cdots &\diamondsuit&\diamondsuit
\end{array}\right)
=\left(\begin{array}{cccccccc}
s_{i-1}&s_{i}&\cdots& s_n&-s_n&\cdots &-s_{i}&-s_{i-1}\\
-s_{i}&\diamondsuit&\cdots &\diamondsuit&\diamondsuit&\cdots &\diamondsuit&\diamondsuit
\end{array}\right).
\end{eqnarray*}

Due to $D_{\mathfrak{p}}$, $D_{\mathfrak{p}}(-s_{i})=s_{i}=-s_{i}+1$ (note that $n+1$ is interpreted as $-n$).
The only situation satisfying this condition is that $s_{i}=-n$, i.e.,
the sign of $n$ in the given signed permutation is negative.
Then, we have $\pi(s_{i-1})=s_{2n-j}=s_{2n+1-i}=n$.
We believe that in this case Lemma~\ref{2lem4} (instead of Lemma~\ref{2lem5}) provides a $2$-reversal.
Namely, $s_{i-1}$ (i.e., the preimage of $n$),~$s_n$ and $n=s_{2n+1-i}$ will form a Case~$2$ transpose
in Lemma~\ref{2lem4}, which will be true if $n$ and $s_n$ are in the same cycle of $\pi$, i.e.,
$\pi$ has a cycle $(s_{i-1}  ~s_{2n+1-i} ~\ldots ~s_n ~\ldots)$.
In order to illustrate this we consider
\begin{example}
\begin{eqnarray*}
\left(\begin{array}{ccccccccc}
0&-3&1&2&-4&4&-2&-1&3\\
-4&0&1&4&3&-3&-2&2&-1
\end{array}\right) \Longrightarrow &\pi =&(0 ~-4 ~3 ~-1 ~2 ~4 ~-3)(1)(-2)\\
\left(\begin{array}{ccccccccc}
0&2&-4&-1&3&-3&1&4&-2\\
1&4&-2&2&-4&0&3&-3&-1
\end{array}\right)\Longrightarrow &\pi =&(0 ~1 ~3 ~-4 ~-2 ~-1 ~2 ~4 ~-3)
\end{eqnarray*}
We inspect, that in the first case $s_{i-1}=2$,~$s_n=-4$ and $n=4$ form a Case~$2$ transpose of Lemma~\ref{2lem4}.
In the second case $s_{i-1}=2$,~$s_n=3$ and $n=4$ form again a Case~$2$ transpose of Lemma~\ref{2lem4}.
\end{example}

Therefore, we conjecture
\begin{conjecture}\label{7conj1}
Let $\mathfrak{p}=(\tilde{s},\pi)$ be exact, skew-symmetric on $[n]^*\cup[n]^{-}$
where $D_{\mathfrak{p}}=p_r^{-1}$ and suppose $\pi(s_{i-1})=s_{2n+1-i}$, where $1\leq i\leq n$.
Then, $n$ and $s_n$ are in the same cycle of $\pi$.
\end{conjecture}

Lemma~\ref{7lem1}, Conjecture~\ref{7conj1} and an analysis of the preservation of exactness
under $2$-reversals suggest, that for a random signed permutation, it is likely to be
possible to transform $s$ into $e_n^{\natural}$ via a sequence of $2$-reversals.
In fact, many examples, including Braga~\cite[Table $3.2$]{braga}, indicate that the lower bound of
Theorem~\ref{7thm1} gives the exact reversal distances.

Note that the lower bound obtained in~\cite{pev2,pev4} via the break point graph also provides
the exact reversal distance for most of signed permutations although the exact reversal distance
was formulated later in~\cite{pev3}.
Now we give a brief comparison of our formula Eq.~\eqref{5eq1} and the lower bound via break point graph.
The break point graph for a given signed permutation $a=a_1a_2\cdots a_n$ on $[n]$ can be obtained as follows:
replacing $a_i$ with $(-a_i) a_i$, and adding $0$ at the beginning of the obtained sequence while
adding $-(n+1)$ at the end of the obtained sequence,
in this way we obtain a sequence $b=b_0b_1b_2\cdots b_{2n}b_{2n+1}$ on $[n]^* \cup [n+1]^{-}$.
Draw a black edge between $b_{2i}$ and $b_{2i+1}$, as well as a grey edge between $i$ and $-(i+1)$
for $0\leq i \leq n$. The obtained graph is the break point graph $BG(a)$ of $a$.
Note that each vertex in $BG(a)$ has degree two so that it can be decomposed into disjoint cycles.
Denote the number of cycles in $BG(a)$ as $C_{BG}(a)$. Then, the lower bound via the break point graph is
\begin{align}\label{bg-r}
d_r(a) \geq n+1-C_{BG}(a).
\end{align}

On the permutation group theory side, we refer the readers to~\cite{feimei,Meid} for discussion of
the reversal distance. Algebraically, we can express Eq.~\eqref{bg-r} in a form similar to our lower bound.
Let $\theta_1$,~$\theta_2$ be the two involutions (without fixed points) determined by
the black edges and grey edges in the break point graph, respectively, i.e.,
\begin{align*}
\theta_1 &=(b_0 ~b_1)(b_2 ~b_3)\cdots (b_{2n} ~b_{2n+1}),\\
\theta_2 &=(0 ~-1)(1 ~-2)\cdots (n ~-n-1).
\end{align*}
It is not hard to observe that $C_{BG}(a)=\frac{C(\theta_1 \theta_2)}{2}$.
Therefore, we have
\begin{proposition}
\begin{align}\label{bp-r2}
d_r(a)\geq \frac{2n+2- C(\theta_1 \theta_2)}{2}.
\end{align}
\end{proposition}

Since both our lower bound Eq.~\eqref{5eq1} and the lower bound Eq.~\eqref{bp-r2} provide
the exact reversal distance for most of signed permutations, that suggests, for most of signed permutations,
$$
C(p_r \tilde{s})= C(\theta_1\theta_2)-1=2 C_{BG}(a)-1.
$$

At the end of the discussion on reversals for signed permutations,
we present the following generalization of Conjecture~\ref{7conj1}:

\begin{conjecture}\label{7conj2}
Let $\mathfrak{p}=(\tilde{s},\pi)$ be skew-symmetric on $[n]^*\cup[n]^{-}$
where $D_{\mathfrak{p}}=p_r^{-1}$.
Then, $n$ and $s_n$ are in the same cycle of $\pi$.
\end{conjecture}


{
	
\section{Conclusion}
In this paper, we studied plane permutations.
We
studied the transposition action on plane permutations obtained by permuting
their diagonal-blocks.
We established basic properties of plane permutations and studied
transpositions and
exceedances and derived various enumerative results. We proved a recurrence
for the number
of plane permutations having a fixed diagonal and $k$ cycles in the
vertical,
generalizing Chapuy's recursion for maps filtered by the genus.

The plane permutation framework has many applications.
As the first application,
 we gave a new combinatorial proof for a result of Zagier and Stanley
by viewing ordinary permutations as a particular class of plane permutations and classify them by their diagonals.
 We combinatorially prove a new recurrence
 satisfied by the unsigned Stirling numbers of the first kind which is 
 the same recurrence for $\xi_{1,k}(n)$ derived from one of the obtained recurrences on 
 plane permutations so that the Zagier-Stanley result follows.
 
Next, motivated by the close connection between swap of segments in sequences and the transposition
action on plane permutations, 
as another application,
we integrated several results on the transposition and
block-interchange distance of permutations as well as the reversal
distance
of signed permutations.
This plane permutation framework is clearly not the same as the approach based on graph models. It is not like any work based on the permutation group theory either, although it uses permutations to discuss and express the results. Hence, we believe that the permutation framework here shed some new insights on these distance problems.
Specifically, we obtained general lower bounds for the transposition distance and the block-interchange distance,
which motivates certain optimization problems. 
The existing lower bounds obtained by Bafna-Pevzner~\cite{pev1}, Christie~\cite{chri2}, Lin et al.~\cite{lin}, Huang et al.~\cite{hhtl}, Labarre~\cite{labarre2}
can be refined by a particular choice of the free parameter in the general lower bound.
As to the reversal distance of signed permutations, we translated it into the block-interchange distances
of skew-symmetric sequences and immediately obtained a lower bound.

In addition plane permutations facilitated to study graph embeddings~\cite{ChR1}, specifically by making local adjustments
to the embeddings. It was crucial to allow for general diagonals in the plane permutations in this context.

As for future directions and outlook, we will study the above mentioned optimization problem for odd and even cycles
(Theorem~\ref{5thm1}). This is not only important in the context of the cycle-graph model or its variations, but also interesting
as a purely combinatorial problem.

Recently, Bura, Chen and Reidys~\cite{bchr} proved that our lower bound for the reversal distance equals the lower bound
obtained by Bafna and Pevzner, as conjectured by one anonymous referee. We furthermore proved Conjecture~\ref{7conj1} and Conjecture~\ref{7conj2}.

}

\begin{acknowledgements}
We thank the anonymous referees for their valuable feedback.
\end{acknowledgements}

\end{document}